\documentclass[12pt]{amsart}  
\usepackage{amsmath,amssymb,latexsym, amsthm, amscd, mathrsfs, stmaryrd,enumerate}
\usepackage[linktocpage=true]{hyperref}
\usepackage{color}
\usepackage[parfill]{parskip}
\usepackage[all]{xy}
\newtheorem{thm}{Theorem} [section]
\theoremstyle{definition}

\newtheorem{example}[thm]{Example}%[section]
\newtheorem{rem}[thm]{Remark}%[section]
\theoremstyle{plain}
\newtheorem{prop}[thm]{Proposition}
\newtheorem{lem}[thm]{Lemma}
\newtheorem{cor}[thm]{Corollary}

\numberwithin{equation}{section}

\newcommand{\ad}{{\text{ad}}}

\newcommand{\C}{\mathbb C}

\newcommand{\g}{\mathfrak{g}}
\newcommand{\gl}{\mathfrak{gl}}
\newcommand{\h}{\mathfrak{h}}

\newcommand{\la}{\lambda}
\newcommand{\mc}{\mathcal}
\newcommand{\mf}{\mathfrak}
\newcommand{\n}{\mathfrak n}

\newcommand{\one}{{\ov 1}}

\newcommand{\ov}{\overline}

\newcommand{\ve}{{\varepsilon}}

\newcommand{\oo}{{\ov 0}}

\newcommand{\coindp}{$\text{Coind}_{\mf{g}_{\geq 0}}^{\mf{g}}$}
\newcommand{\coindn}{$\text{Coind}_{\mf{g}_{\leq 0}}^{\mf{g}}$}

\newcommand{\ucoindp}{$\text{coind}_{\mf{g}_{\geq 0}}^{\mf{g}}$}

\title[Simple supermodules over Lie superalgebras]{Simple supermodules over Lie superalgebras}

\author[Chih-Whi Chen]{Chih-Whi Chen}
\address{} \email{}

\author[Volodymyr Mazorchuk]{Volodymyr Mazorchuk}
\address{}
\email{}

\date{}


\begin{document}

\begin{abstract} 
We show that, for many Lie superalgebras admitting a compatible $\mathbb{Z}$-grading, 
Kac induction functor gives rise to a bijection between simple  
supermodules over a Lie superalgebra  and simple supermodules over 
the even part of this Lie superalgebra. This reduces the classification problem for 
the former to the one for the latter. Our result applies to all
classical Lie superalgebra of type $I$, in particular,
to the general linear Lie superalgebra $\gl(m|n)$. In the latter case we also
show that the rough structure of simple $\gl(m|n)$-supermodules and also that of Kac supermodules depends
only on the annihilator of the $\mf{gl}(m)\oplus \mf{gl}(n)$-input and 
hence can be computed using the combinatorics of BGG category $\mathcal{O}$.
\end{abstract}
\maketitle

\setcounter{tocdepth}{1}
%\tableofcontents

%%%%%%%%%%
% \section{Introduction}

 %
 %
 %\subsection{Background}
 % \subsection{The main results and strategy}
 % \subsection{Organization}
   \vskip 1cm

\section{Introduction and description of the results}\label{sectintro}
 
Classification problems are central in representation theory. One of the basic classification
problems is the problem of classification of all simple modules for a given algebra. For
Lie algebras, this problem is rather difficult. For simple Lie algebras, some kind of solution
(more precisely, a reduction theorem which reduces classification of simple modules to classification of
equivalence classes of irreducible elements in a certain non-commutative principal ideal domain)
exists only for the Lie algebra $\mathfrak{sl}(2)$, see Block's paper \cite{Bl}.

For a Lie superalgebra $\mathfrak{g}$, classification of simple $\mathfrak{g}$-supermodules is,
naturally, at least as hard as classification of simple modules over the even Lie algebra part 
$\mathfrak{g}_{\oo}$. In case $\mathfrak{g}_{\oo}$ is isomorphic to $\mathfrak{sl}(2,\mathbb{C})$
or $\mathfrak{gl}(2,\mathbb{C})$, one could expect some analogue of Block's classification theorem.
For the Lie superalgebras $\mathfrak{osp}(1|2)$, such an analogue was obtained in \cite{BO}
following Block's approach and, for the Lie superalgebra $\mathfrak{q}(2)$ (and its various subquotients), 
such an analogue was obtained in \cite{Ma} using a reduction technique based on application 
of Harish-Chandra bimodules.

There are also some special cases in which much stronger results are known. The most significant one
is the equivalence of certain categories of strongly typical $\mathfrak{g}$-supermodules and 
certain categories of $\mathfrak{g}_{\oo}$-modules established in \cite{Go02s} for basic classical
Lie superalgebras. This equivalence automatically provides a bijection between isomorphism classes of 
simple objects in categories in question and hence reduces the relevant part of the
classification problem for $\mathfrak{g}$ to the corresponding problem for $\mathfrak{g}_{\oo}$.
There are also various constructions of certain classes of simple (non highest weight)
modules over Lie superalgebras, see e.g. \cite{DMP,GG,FGG,BCW,WZZ,BM,CZ} and references therein.

The main motivation for the present paper was to investigate in which generality one can obtain
a complete reduction result which connects classification of simple $\mathfrak{g}$-supermodules and
classification of simple $\mathfrak{g}_{\oo}$-(super)modules. Our first main result is the following.

{\bf Theorem A.} {\em Let $\mf g = \mf g_\oo\oplus \mf g_{\one}$ be a Lie superalgebra 
with $\text{dim}(\mf g_{\one})< \infty$ admitting
a compatible $\mathbb{Z}$-grading $\mf g= \mf g_{-1}\oplus \mf g_0 \oplus \mf g_{1}$.
Let $K({}_-):\mathfrak{g}_\oo\text{-}\mathrm{smod}\to \mathfrak{g}\text{-}\mathrm{smod}$ 
be Kac induction functor. Then, for any simple $\mathfrak{g}_\oo$-supermodule $V$,
the $\mathfrak{g}$-supermodule $K(V)$ has simple top, denoted $L(V)$, and the
correspondence $V\mapsto L(V)$ gives rise to a bijection between the sets of 
isomorphism classes of simple $\mathfrak{g}_\oo$- and $\mathfrak{g}$-supermodules.}

We also provide, in full generality, several criteria for simplicity of Kac modules,
with arbitrary simple input, in terms of typicality of the involved central characters.
In the case of the general linear Lie superalgebra $\mathfrak{gl}(m|n)$ we also
study the rough structure of Kac modules with arbitrary simple input along with the
$\mathfrak{g}_\oo$-rough structure of simple $\mathfrak{g}$-supermodules.
Our second main result is the following:

{\bf Theorem B.} {\em For $\mathfrak{g}=\mathfrak{gl}(m|n)$, let $L$ and $L'$ be two
simple $\mathfrak{g}$-supermodules such
that there is a finite-dimensional supermodule $E$ with $E\otimes L\twoheadrightarrow L'$. 
Let $V$ and $V'$ be their $\mathfrak{g}_\oo=\mathfrak{gl}(m)\oplus \mathfrak{gl}(n)$-cor\-res\-pon\-dents.
Then both multiplicities $[K(V):L']$ and  
$[\mathrm{Res}^{\mathfrak{g}}_{\mathfrak{g}_\oo}(L):V']$ are well-defined and finite
and can be computed using BGG category $\mathcal{O}$. }

The paper is organized as follows: in Section~\ref{sectPre} we collected all necessary preliminaries.
In Section~\ref{s3} we compare (induced) Kac modules with their coinduced counterparts.
Section~\ref{s4} studies simple supermodules and, in particular, contains a proof of Theorem~A.
In this section one can also find detailed examples of all classical Lie superalgebras of type I
and various criteria of simplicity for Kac modules. Section~\ref{s5} is devoted to the study 
of rough structure and, in particular, establishes Theorem~B.

\vspace{5mm}
{\bf Acknowledgment.}
The first author is supported by Vergstiftelsen. The second author is supported by
the Swedish Research Council and G{\"o}ran Gustafsson Stiftelser.

\section{Preliminaries}\label{sectPre}
 
\subsection{}\label{sectPre.1} Throughout this paper, we let $\mf g = \mf g_\oo\oplus \mf g_{\one}$ be a Lie superalgebra with $\text{dim}( \mf g_{\one})< \infty$ and assume that $\mf g$ has a compatible $\mathbb{Z}$-grading $\mf g= \mf g_{-1}\oplus \mf g_0 \oplus \mf g_{1}$. Namely, $\mf g_0 =\mf g_\oo$ and  $\mf g_\one = \mf g_1 \oplus \mf g_{-1}$, where $\mf g_{\pm 1}$ are $\mf g_\oo$-submodules of $\mf g_\one$ with $[\mf g_{1},\mf g_{ 1}] =[\mf g_{-1},\mf g_{-1}]  =0$. We set $\mf g_{\geq 0}: = \mf g_0 \oplus \mf g_1$ and $\mf g_{\leq 0}: = \mf g_0 \oplus \mf g_{-1}$.
   
\subsection{} For a given vector superspace $V = V_\oo \oplus V_\one$ and a homogeneous element $v\in V$, we denote the parity of $v$ by $\ov v$. 
    We recall the parity reversing functor $\Pi$ defined on the category of vector superspaces as follows:
   $$(\Pi V)_\oo =V_\one,~(\Pi V)_\one =V_\oo.$$
   Throughout the present paper, all homomorphisms in the category of modules over Lie superalgebras are supposed to be homogeneous of degree zero. 
   Therefore, a module $M$ over a Lie superalgebra is not necessary isomorphic to  $\Pi M$.

    \subsection{}For a given Lie (super)algebra $\mf L$, we denote the universal enveloping algebra of $\mf L$ by $U(\mf L)$.   We let $U = U(\mf g)$ and $U_{\oo} =U(\mf g_\oo)$. Observe that $U$ is a finite extension of the ring $U_{\oo}$ with basis $\Lambda(\mf g_{\bar{1}})$, the exterior algebra of the vector space $\mf g_{\bar{1}}$. Let $Z(\mf g)$ and $Z(\mf g_\oo)$ denote the center of $U$ and $U_\oo$, respectively. Also, we denote the center of $\mf g_\oo$  by $\mf z(\mf g_\oo)$.
   For a given $\mf g$- (resp. $\mf g_\oo$-) central character $\chi$ and a $\mf g$- (resp. $\mf g_\oo$-)
   module $M$, we set $$M_{\chi} : = \{m\in M \|~ (z-\chi(z))^rm =0, \text{ for all } z\in Z(\mf g) 
   \,\,(\text{resp. }  z\in Z(\mf g_\oo))\text{ and all }r\gg 0 \}.$$ 
   
Consider the category $\mf g$-smod $= U$-smod of finitely generated (left) $U$-supermodules, 
the category $\mf g_{\oo}$-smod $= U_\oo$-smod of finitely generated (left) $U_\oo$-supermodules, and the category $U$-mod-$U$ of finitely generated $U$-$U$-bimodules.
As $\mf g_{\oo}$ is even, $\mf g_{\oo}\text{-smod}$ is just a direct sum of two copies of 
$\mf g_{\oo}\text{-mod}$, the category of finitely generated (left) $U_\oo$-supermodules.
We have the exact restriction, induction and coinduction functors 
\begin{displaymath}
\text{Res}_{\mf g_\oo}^{\mf g}:\mf g \text{-smod}\to \mf g_{\oo}\text{-smod}\quad\text{ and }\quad
\text{Ind}_{\mf g_\oo}^{\mf g},~\text{Coind}_{\mf g_\oo}^{\mf g}: \mf g_{\oo}\text{-smod}\rightarrow  \mf g \text{-smod}.
\end{displaymath}
%$$\text{Res}_{\mf g_\oo}^{\mf g},~\text{Ind}_{\mf g_\oo}^{\mf g},~\text{Coind}_{\mf g_\oo}^{\mf g}: \mf g_{\oo}\text{-smod}\rightarrow  \mf g \text{-smod}.$$%\qquad\text{ and }\qquad \text{Ind}_{\mf g_\oo}^{\mf g}:\mf g_{\oo}\text{-smod}\rightarrow \mf g\text{-smod}.$$
%By \cite[Theorem~3.2.3]{Go}, see also \cite{BF}, these functors are biadjoint up to parity change.  If $\dim(\mf g_{\one})$ is even,they are biadjoint. In particular, this holds if $\mf g$ is one of the basic classical Lie superalgebras. 
By \cite[Theorem 2.2]{BF} (also see \cite{Go}), the functors
$\text{Ind}_{\mf g_\oo}^{\mf g}$ and $\text{Coind}^{\mf g}_{\mf g_\oo}$ are isomorphic up to the equivalence given by tensoring with the one-dimensional $\mf g_\oo$-module on the
top degree subspace of $U(\mf g_\one) = \Lambda \mf g_{\one}$.    
   
For a $\mf g$-central character $\chi$, we denote by $\mf g\text{-smod}_{\chi}$ the full subcategory of $\mf g$-smod consisting of all $\mf g$-supermodules annihilated by some power of $\chi$. Similarly, for a $\mf g_{\oo}$-central character $\chi$, we denote by $\mf g_{\oo}\text{-smod}_{\chi}$ the full subcategory of $\mf g_{\oo}$-smod consisting of all $\mf g_{\oo}$-supermodules annihilated by some power of $\chi$.

   \subsection{Induced modules} \label{sectInd} 
   For a given $\mf g_\oo$-supermodule $V$, we
   may extend $V$ trivially to a $\mf g_\oo \oplus \mf g_{1}$-supermodule  and define the {\em Kac module} of $V$ as follows: $$K(V) :=
   \text{Ind}_{\mf g_{\geq 0}}^{\mathfrak{g}}(V).$$

   This defines an exact functor $K(\cdot): \mf g_\oo\text{-smod}\rightarrow \mf g\text{-smod}$ which we call {\em Kac functor}.
   For a given $M\in \mf g$-smod, we have the usual adjunction
   \begin{equation}\label{eqkac1}
   \text{Hom}_{\mf g}(K(V), M)  = \text{Hom}_{\mf g_\oo}(V, M^{\mf g_1}),  
   \end{equation}
   where $M^{\mf g_1}:=\{m\in M\| ~ \mf g_1\cdot m =0\}.$
   
      Also, for a given  $\mf g_\oo$-supermodule $V$, we define the {\em opposite Kac module} $K'(V)$ of $V$, see e.g. \cite[Section 3.3]{Ge98}, which is given as follows:
   \begin{align}\label{DefoppoKac} &K'(V): = \text{Ind}_{\mf g_{\leq 0}}^{\mf g}(V),\end{align}
   where ${\mf g_{-1}}V$ is defined to be zero. Just like in the previous paragraph, this defines an exact functor $K'(\cdot): \mf g_\oo\text{-smod}\rightarrow \mf g\text{-smod}$ and we have a similar adjunction 
   \begin{equation}\label{eqkac2}
   \text{Hom}_{\mf g}(K'(V), M)  = \text{Hom}_{\mf g_\oo}(V, M^{\mf g_{-1}}),  
   \end{equation}
   where $M^{\mf g_{-1}}:=\{m\in M\| ~ \mf g_{-1}\cdot m =0\}.$
   We may observe that $K(V) \cong \Lambda(\mf g_{-1}) \otimes V$ and $K'(V) \cong \Lambda(\mf g_{1}) \otimes V$ as vector spaces. 
   
We set 
\begin{displaymath}
\Lambda^{\text{max}}(\mf g_{-1}) :=\Lambda^{\text{dim}\mf g_{-1}}(\mf g_{-1}) \quad\text{ and }\quad
\Lambda^{\text{max}}(\mf g_{1}) :=\Lambda^{\text{dim}\mf g_{1}}(\mf g_{1}).
\end{displaymath}

   \subsection{Coinduced modules} \label{sectCoind}  
    For a given $\mf g_\oo$-supermodule $V$, we
   may extend $V$ trivially to a $\mf g_\oo \oplus \mf g_{1}$-mo\-du\-le and define the {\em super} coinduced module  $\text{{\coindp $(V)$}}$ (cf. \cite{Ge98} and \cite[Chapter 4, Section 2]{Sc06}) of $V$ as follows:
   \begin{align*}&\{ f\in \text{Hom}_{\mathbb{C}}(U(\mathfrak{g}),N) |~f(pu) =(-1)^{\overline{p}\overline{f}} pf(u), \text{ for all homogeneous }p\in \mf{g}_{\geq 0}, u\in U(\mf{g})\}, \end{align*}
   with the action $(xf)(u):=(-1)^{\overline{x}(\overline{u}+\overline{f})}f(ux)$, for all homogeneous $x\in \mathfrak{g}$, $f\in \text{{\coindp $(V)$}}$ and $u\in U(\mf{g})$. In particular, we may observe that $\text{{\coindp $(V)$}}\cong \text{Hom}_{\mathbb C}(\Lambda (\mf g_{-1}), V)$ as vector space.
   
   The following adjunction is proved in \cite[Chapter 4, Section~2, Proposition 3]{Sc06}.
   \begin{lem}\label{supercoind} There are natural isomorphisms
   	$$\emph{Hom}_{\mf g_{\geq 0}}(\emph{Res}_{\mf g_{\geq 0}}^{\mf g}(V),W) \cong \emph{Hom}_{\mf g}(V,\emph{{\coindp $(W)$}})$$ given by 
   	\begin{align}
   	&  \emph{Hom}_{\mf g_{\geq 0}}(\emph{Res}_{\mf g_{\geq 0}}^{\mf g}(V),W) \ni \phi \mapsto \hat\phi \in  \emph{Hom}_{\mf g}(V,\emph{{\coindp $(W)$}}),\\
   	& \emph{Hom}_{\mf g_{\geq 0}}(\emph{Res}_{\mf g_{\geq 0}}^{\mf g}(V),W) \ni \tilde\psi  \mapsfrom \psi \in  \emph{Hom}_{\mf g}(V,\emph{{\coindp $(W)$}}),
   	\end{align}	where $\hat\phi(v)(y): = (-1)^{\ov y\ov v}\phi(yv),$
   	and $\tilde\psi(v): = \psi(v)(1),$ for homogeneous $y\in U$ and $v\in V$.
   \end{lem}

Also we define the  {\em usual}  coinduced module $\text{{\ucoindp $(V)$}}$ (cf. \cite{Go}) of $V$ as the following $\mf g$-supermodule:
\begin{align*}&\{ f\in \text{Hom}_{\mathbb{C}}(U(\mathfrak{g}),N) |~f(pu) = pf(u), \text{ for all homogeneous }p\in \mf{g}_{\geq 0}, u\in U(\mf{g})\}, \end{align*}
with the action $(xf)(u):=f(ux)$, for all $x\in \mathfrak{g}$, $f\in \text{{\ucoindp $V$}}$ and $u\in U(\mf{g})$. 

Also, we have the following adjunction between the restriction functor and usual coinduction functor.

\begin{lem}\label{usualcoind} 
There are natural isomorphisms
$$\emph{Hom}_{\mf g_{\geq 0}}(\emph{Res}_{\mf g_{\geq 0}}^{\mf g}(V),W) = 
\emph{Hom}_{\mf g}(V,\emph{{\ucoindp $(W)$}})$$ given by
\begin{align}
&  \emph{Hom}_{\mf g_{\geq 0}}(\emph{Res}_{\mf g_{\geq 0}}^{\mf g}(V),W) \ni \phi \mapsto \hat\phi \in  \emph{Hom}_{\mf g}(V,\emph{{\ucoindp $(W)$}}),\\
& \emph{Hom}_{\mf g_{\geq 0}}(\emph{Res}_{\mf g_{\geq 0}}^{\mf g}(V),W) \ni \tilde\psi  \leftarrow \psi \in  \emph{Hom}_{\mf g}(V,\emph{{\ucoindp $(W)$}}),
\end{align}	where $\hat\phi(v)(y): = \phi(yv),$
and $\tilde\psi(v): = \psi(v)(1),$  for $y\in U$ and $v\in V$.
\end{lem}

\begin{proof}
This follows from the usual adjunction between restriction and coinduction.
\end{proof}

Since both super coinduction functor \coindp$(\bullet)$ and usual coinduction functor $\text{coind}_{\mf g_{\geq} 0}^{\mf g}(\bullet)$ are right adjoint to the restriction functor by Lemma \ref{supercoind} and Lemma \ref{usualcoind}, it follows that \coindp$(\bullet) \cong \text{coind}_{\mf g_{\geq 0}}^{\mf g}(\bullet)$. The following lemma gives an explicit $\mf g$-isomorphism.

\begin{lem} \label{lem::CoindareIso}
There is a natural isomorphism {\em \coindp $(V)$} $\cong$ {\em \ucoindp $(V)$}.
\end{lem}
 
\begin{proof}
Define a linear isomorphism $(\bullet)^{\dag}: \text{{\coindp $(V)$} $\rightarrow $ {\ucoindp $(V)$}}$ via 
$$f^{\dag}(y): = (-1)^{\ov f \cdot \ov y}f(y),$$
for all homogeneous elements $f \in \text{{\coindp $(V)$}}$ and $y \in U$ (also see, e.g., \cite[Definition 1.1]{ChWa12}).　 For a given  $f \in \text{{\coindp $(V)$}}$,  we check that $f^{\dag}\in$ {\ucoindp (V)}:
\begin{align}
&f^{\dag}(yx) = (-1)^{\ov f(\ov y+\ov x)}f(yx) = (-1)^{\ov f(\ov x)}yf(x) =y f^{\dag}(x),
\end{align} for homogeneous $y\in U(\mf g_{\geq 0})$ and $x \in U$. Therefore $(\bullet)^{\dag}$ is an even linear isomorphism between superspaces. 
 	
We now show that $(\bullet)^{\dag}$ intertwines the $\mf g$-actions. Let $f\in \text{{\coindp $(V)$}}$ and $x\in \mf g$ be homogeneous elements.  We may note that the parity of $xf$ is $\ov x + \ov f$. Also, for a given homogeneous element $y\in U$, we have 
\begin{align*}
(xf)^{\dag}(y)
& =(-1)^{(\ov x+\ov f)\ov y} (xf)(y) \\
& =(-1)^{(\ov x+\ov f)\ov y}(-1)^{(\ov f+\ov y)\ov x}f(yx) \\
& =(-1)^{(\ov x+\ov y)\ov f}f(yx) \\
& =f^{\dag}(yx)  = (xf^{\dag})(y). 
\end{align*} 
Consequently, we have $(xf)^{\dag} = xf^{\dag}$, for all $x\in \mf g$. 
Therefore $(\bullet)^{\dag}$ is an even isomorphism of $\mf g$-supermodules. The naturality of
$(\bullet)^{\dag}$ follows directly from the definitions.
\end{proof}

We note that Lemma \ref{lem::CoindareIso} can be used to match adjunctions in 
Lemma \ref{supercoind} and Lemma \ref{usualcoind}.

In a similar way, one  defines the {\em opposite} coinduced modules \coindn$(V)$, for a simple $\mf g_\oo$-supermodule $V$.
   
\section{Induced vs coinduced modules}\label{s3}

In this section we study, in detail, induced modules and coinduced modules and relations between them.
        
\subsection{Kac modules and opposite Kac modules}

\begin{lem}\label{newl1} Let $V$ be any non-zero $\mf g_0$-supermodule. Then
every non-zero $\mathfrak{g}_{-1}$-submodule of $K(V)$ has a non-zero intersection with 
$\Lambda^{\emph{max}}(\mf g_{-1})\otimes V$, in particular,
\begin{displaymath}
(K(V))^{\mathfrak{g}_{-1}}=\Lambda^{\emph{max}}(\mf g_{-1})\otimes V.
\end{displaymath}
\end{lem}

\begin{proof} We first fix a basis $B$ for $\mf g_{-1}$. 
Recall that $K(V) \cong \Lambda(\mf g_{-1}) \otimes V$ as a vector space. 
For a given $0\leq j\leq  \text{dim}\mf g_{-1}$, let $\{x_1,\ldots ,x_{\ell}\}$  be a basis of 
$\Lambda^{j}(\mf g_{-1})$ such that each $x_i$ is a monomial in $B$. Then, for each 
$x_i$, there is an element $y_k \in \Lambda(\mf g_{-1})$ such that 
$y_k\cdot x_i \in \delta_{ik}(\Lambda^{\text{max}}(\mf g_{-1})\backslash \{0\})$. 
From this it follows that, for a given non-zero element $x\in K(V)$, we have 
$\Lambda(\mf g_{-1})x \cap \big(\Lambda^{\text{max}}(\mf g_{-1})\otimes V\big) \neq 0$.
The claim follows.
\end{proof}

\begin{lem}\label{Kacsimplesoc} Let $V$ be a simple $\mf g_0$-supermodule. Then
every $\mf g$-submodule of $K(V)$ contains $\Lambda^{\emph{max}}(\mf g_{-1})\otimes V$, 
in particular,  $$\emph{Soc}(K(V)) = U \cdot(\Lambda^{\emph{max}}(\mf g_{-1})\otimes V)$$ 
and, moreover, $\emph{Soc}(K(V))$ is a simple module.
Similarly, $K'(V)$ has simple socle. In particular, both $K(V)$ and $K'(V)$ are indecomposable.
\end{lem}

\begin{proof}
Since $\Lambda^{\text{max}}(\mf g_{-1})\otimes \Lambda^{\text{max}}(\mf g_{-1}^*)$ 
is isomorphic to the trivial $\mf g_0$-supermodule, we obtain that 
$\Lambda^{\text{max}}(\mf g_{-1})\otimes V$ is a simple $\mf g_\oo$-supermodule. 
Now the claim follows directly from Lemma~\ref{newl1}.
\end{proof}

The observation of Lemma~\ref{Kacsimplesoc} that Kac modules have simple socles is interesting 
and slightly unexpected. In would be natural to expect that Kac modules, being {\em induced}, have
simple tops. The latter statement, however, requires much more effort and we refer the
reader to Theorem \ref{MainThm1}. 

\begin{rem} \label{InvInKac} 
By Lemma~\ref{newl1}, for a simple $\mf g_0$-supermodule $V$, we have $$K(V)^{\mf g_{-1}} = \Lambda^{\text{max}}(\mf g_{-1})\otimes  V$$ and, similarly, that $K'(V)^{\mf g_{1}} = 
\Lambda^{\text{max}}(\mf g_1 )\otimes V$. Therefore  we have
\begin{align} 
&(\text{Soc}(K(V)))^{\mf g_{-1}}  = \Lambda^{\text{max}}(\mf g_{-1}) \otimes V, \label{eqn1}
\\ &(\text{Soc}(K'(V)))^{\mf g_{1}}  = \Lambda^{\text{max}}(\mf g_{1}) \otimes V.\label{eqn2}
\end{align} 
In particular, since $\Lambda^{\text{max}}(\mf g_i^*) \otimes \Lambda^{\text{max}}(\mf g_i)$ are trivial $\mf g_0$-supermodules, for $i=\pm 1$, we have 
$$\text{Soc}(K(V)) \cong \text{Soc}(K(W)) \quad\text{ if and only if }\quad V\cong W$$
and 
$$\text{Soc}(K'(V)) \cong \text{Soc}(K'(W)) \quad\text{ if and only if }\quad V\cong W.$$
\end{rem}
 
\begin{cor} \label{KacOppoKac} 
Let $V$ and $W$ be simple $\mf g_0$-supermodules. Then we have:
\begin{enumerate}[$($a$)$]
\item \label{KacOppoKac.1} Existence of a non-zero homomorphism from $K(V)$ to $K'(W)$
is equivalent to the condition $V \cong \Lambda^{\emph{max}}(\mf g_1) \otimes W$.
\item \label{KacOppoKac.2} Every non-zero $\mf g$-homomorphism $f:K(V)\rightarrow K'(W)$ satisfies $f(K(V)) =\emph{Soc}(K'(W))$.
\item \label{KacOppoKac.3} If $K(V)\cong K'(W)$, then both $K(V)$ and $K'(W)$ are simple supermodules. 
\end{enumerate}
\end{cor}

\begin{proof}
Using  \eqref{eqkac1} and \eqref{eqn2}, we have
\begin{align}
&\text{Hom}_{\mf g}(K(V), K'(W))  = \text{Hom}_{\mf g_\oo}(V, K'(W)^{\mf g_1}) = \text{Hom}_{\mf g_\oo}(V, \Lambda^{\text{max}}(\mf g_1) \otimes W).
\end{align} 
As $\Lambda^{\text{max}}(\mf g_1) \otimes W$ is a simple $\mf g_\oo$-supermodule, claim~\eqref{KacOppoKac.1} follows.
From claim~\eqref{KacOppoKac.1} and
an analogue of Lemma \ref{Kacsimplesoc} for $K'(W)$ we also obtain claim~\eqref{KacOppoKac.2}.
Finally, claim~\eqref{KacOppoKac.3} follows from claim~\eqref{KacOppoKac.2}.
\end{proof}  

We note that one can characterize all isomorphisms between Kac modules and opposite Kac modules
using Corollary \ref{KacOppoKac} and the criteria of simplicity of Kac modules given in 
Section~\ref{Sect::CriforSimpleKac}.

   \vskip 0.5cm
   %\subsection{$\text{Coind}_{\mf g_{\geq 0}}^{\mf g}(\Lambda^{\text{max}} \mf g_{-1}\otimes \bullet)\cong \text{Ind}_{\mf g_{\geq 0}}^{\mf g}(\bullet)$}
\subsection{Isomorphism between induced and coinduced modules}
   
The following theorem is an analog of \cite[Proposition 2.1.1(ii)]{Ge98}.
   
\begin{thm} \label{Thm::CoindisoInd} 
For a given simple $\mf g_0$-supermodule $V$, we have	
\begin{align}
&\emph{Coind}_{\mf g_{\geq 0}}^{\mf g}(\Lambda^{\emph{max}}(\mf g_{-1})\otimes V)\cong K(V), \label{Eq::CoindisoInd1} \\
&\emph{Coind}_{\mf g_{\geq 0}}^{\mf g}(V)\cong K(\Lambda^{\emph{max}}(\mf g_{-1}^*) \otimes V), \label{Eq::CoindisoInd2}
\end{align} up to parity change.
Similarly, we have 
\begin{align}
&\emph{Coind}_{\mf g_{\leq 0}}^{\mf g}(\Lambda^{\emph{max}}(\mf g_{1})\otimes V)\cong K'(V), \label{Eq::CoindisoInd3} \\
&\emph{Coind}_{\mf g_{\leq 0}}^{\mf g}(V)\cong K'(\Lambda^{\emph{max}}(\mf g_{1}^*) \otimes V), \label{Eq::CoindisoInd4}
\end{align} up to parity change.
\end{thm} 
   
\begin{proof} 
Note that \eqref{Eq::CoindisoInd1} and \eqref{Eq::CoindisoInd2} are equivalent due to 
the fact that $\Lambda^{\text{max}}(\mf g^*_{-1})\otimes \Lambda^{\text{max}}(\mf g_{-1})$ is 
isomorphic to the trivial $\mf g_0$-supermodule.

By \cite[Theorem~2.2]{BF}, $\mathrm{Coind}_{\mf g_{\geq 0}}^{\mf g}(V)$ is
isomorphic to $K(W)$, for some simple $\mf g_0$-supermodule $W$.
To identify $W$ it is convenient to look at the category $\mathfrak{g}\text{-}\mathrm{mod}^{\mathbb{Z}}$
of all $\mathbb{Z}$-graded $\mathfrak{g}$-supermodules. By construction both $K(X)$ and 
$\mathrm{Coind}_{\mf g_{\geq 0}}^{\mf g}(X)$ are $\mathbb{Z}$-graded, for any
$\mf g_0$-supermodule $X$ concentrated in a single $\mathbb{Z}$-degree.
Note that simple $\mf g_0$-supermodules are always concentrated in a single $\mathbb{Z}$-degree.

Abusing notation, we consider the {\em standard} graded lifts $K(X)$ and $\mathrm{Coind}_{\mf g_{\geq 0}}^{\mf g}(X)$  
in which $X$ is concentrated in degree $0$. Then all non-zero components of $K(X)$ have non-positive degrees
with $\Lambda^{\text{max}}(\mathfrak{g}_{-1}) \otimes X$ concentrated in degree $-\dim(\mathfrak{g}_{-1})$.
All non-zero components of $\mathrm{Coind}_{\mf g_{\geq 0}}^{\mf g}(X)$ have non-negative degrees
with $\dim(\mathfrak{g}_{-1})$ being the maximal non-zero degree.
Therefore $W$ must be isomorphic to the degree $\dim(\mathfrak{g}_{-1})$ component of 
$\mathrm{Coind}_{\mf g_{\geq 0}}^{\mf g}(V)$. By construction
(cf. \cite[Theorem~2.3]{BF}), the latter one is $\Lambda^{\mathrm{max}}(\mf g_{-1}^*) \otimes V$.

This proves isomorphisms \eqref{Eq::CoindisoInd1} and \eqref{Eq::CoindisoInd2}.
Isomorphisms \eqref{Eq::CoindisoInd3} and \eqref{Eq::CoindisoInd4} are proved in a similar way.
\end{proof}

\begin{cor} \label{Coindsimplesoc}
Let $V$ be a simple $\mf g_0$-supermodule.
Then both $\emph{Coind}_{\mf g_{\geq 0}}^{\mf{g}}(V)$ 
and $\emph{Coind}_{\mf{g}_{\leq 0}}^{\mf{g}}(V)$ have simple socle.
\end{cor}

\begin{proof} 
The claim follows from Theorem \ref{Thm::CoindisoInd} 
and Lemma~\ref{Kacsimplesoc}.
\end{proof}

\section{Simple modules over Lie superalgebra of type I}\label{s4}
  
\subsection{Classification of simple $\mf g$-supermodules}
In this subsection we prove that Kac functor gives rise to a one-to-one correspondence between simple $\mf g$-supermodules and simple $\mf g_\oo$-su\-per\-mo\-du\-les. 
The main theorem in this subsection is the following.
  
\begin{thm} \label{MainThm1}
Let $\mf g$ be as in Subsection~\ref{sectPre.1}.
\begin{enumerate}[$($i$)$]
\item \label{MainThm1.1} For any simple $\mf g_\oo$-supermodule $V$, the module $K(V)$
has a unique maximal submodule. The unique simple top of $K(V)$ is denoted $L(V)$.
The correspondence 
\begin{equation}\label{Eqcortop}
V\mapsto L(V) 
\end{equation}
gives rise to a bijection between the set of isomorphism classes
of simple $\mf g_\oo$-supermodules and the set of isomorphism classes
of simple $\mf g$-supermodules.
\item \label{MainThm1.2} The correspondence, \begin{align} &V\mapsto \emph{Soc}\big(\emph{\coindp} (V)\big) \cong \emph{Soc}\big(K(V\otimes \Lambda^{\emph{max}}(\mf g_{-1}^*))\big), \label{Eqcorsoc} \end{align} gives rise to a 
bijection between the set of isomorphism classes of simple $\mf g_\oo$-supermodules 
and the set of isomorphism classes of simple $\mf g$-supermodules.
\end{enumerate}
\end{thm}
  
We need some preparation before we can prove Theorem~\ref{MainThm1}.  
  
\begin{lem} \label{lem::SocleInCoind}
Let $M$ be a simple $\mf g$-supermodule. Then there is a simple $\mf g_0$-supermodule $N$ 
such that $$M \hookrightarrow \emph{Coind}_{\mf{g}_{\geq 0}}^{\mathfrak{g}}(N).$$ 
\end{lem}
  
\begin{proof}
As $U(\mf g)$ is finitely-generated over $U(\mf g_{\geq 0})$,  we have 
$$0\neq \text{Hom}_{\mf g_{\geq 0}}(\text{Res}_{\mf g_{\geq 0}}^{\mf g}(M), N) = \text{Hom}_{\mf g}(M, \text{Coind}_{\mf{g}_{\geq 0}}^{\mathfrak{g}}(N)),$$
for some simple $\mf g_{\geq 0}$-supermodule $N$. 
Since $U(\mf g_1)=\Lambda (\mf g_1)$ is finite dimensional, we have $N^{\mf g_1}\neq 0$. 
Also, $N^{\mf g_1}$ is a $\mf g_{\geq 0}$-submodule and therefore $N =N^{\mf g_1}$. The claim follows. 
\end{proof}

\begin{cor} \label{SocIndCoind}
For a simple $\mf g$-supermodule $M$, there exist simple $\mf g_0$-supermodules 
$V_1$, $V_2$, $V_3$, $V_4$ such that $$M = \emph{Soc}(K(V_1))= \emph{Soc}(K'(V_2))=
\emph{Soc}(\emph{Coind}_{\mf g_{\geq 0}}^{\mf g}(V_3)) = 
\emph{Soc}(\emph{Coind}_{\mf g_{\leq 0}}^{\mf g}(V_4)).$$ 
\end{cor}
  
\begin{proof} 
Since  $\Lambda^{\text{max}}(\mf g_{i}^*)\otimes\Lambda^{\text{max}}(\mf g_{i})$ 
is isomorphic to the trivial $\mf g_0$-supermodule, for $i=\pm 1$, the claim follows 
directly from Lemmata~\ref{lem::SocleInCoind} and \ref{Kacsimplesoc} and Theorem~\ref{Thm::CoindisoInd}.
\end{proof}

\begin{lem}\label{nlem5}
Under the assumptions of Theorem~\ref{MainThm1}, the module  $K(V)$ has a unique maximal submodule.
\end{lem}

\begin{proof}
We first show that all simple quotients of $K(V)$ are isomorphic. Let $f: K(V) \twoheadrightarrow L$ 
be such that $L$ is simple. By Corollary~\ref{SocIndCoind} there is a simple $\mf g_\oo$-supermodule $W$ 
such that $L\cong \text{Soc}(K'(W))$. We would like to show that $W$ is uniquely determined by $V$. 
By Remark~\ref{InvInKac}, we have $f(V)\subseteq \Lambda^{\text{max}}(\mf g_{1})\otimes W$ since 
$V \subseteq  K(V)^{\mf g_1}$. Now $\Lambda^{\text{max}}(\mf g_{1})\otimes W$ is simple since 
$\Lambda^{\text{max}}(\mf g^*_{1})\otimes \Lambda^{\text{max}}(\mf g_{1})$ is isomorphic 
to the trivial $\mf g_0$-supermodule. Thus 
$f|_V: V \cong \Lambda^{\text{max}}(\mf g_{1})\otimes W$, and hence
$W\cong \Lambda^{\text{max}}(\mf g_{1}^*)\otimes  V$ as $\mf g_\oo$-supermodules. 
In other words, every simple quotient of $K(V)$ is isomorphic to the simple socle of 
$K'(\Lambda^{\text{max}}(\mf g_{1}^*)\otimes  V)$.

Using adjunction and Schur's lemma, we also have 
\begin{displaymath}
\dim\mathrm{Hom}_{\mathfrak{g}}(K(V),K'(W))=
\dim\mathrm{Hom}_{\mathfrak{g}_{\geq 0}}(V,\Lambda^{\text{max}}(\mf g_{1})\otimes W)=
\dim\mathrm{Hom}_{\mathfrak{g}_{\geq 0}}(V,V)=1.
\end{displaymath}
The claim of the lemma follows.
\end{proof}

\begin{proof}[Proof of Theorem \ref{MainThm1}] 
We start by proving claim~\eqref{MainThm1.2}. 
Recall that $\Lambda^{\text{max}}(\mf g_{i}^*)\otimes \Lambda^{\text{max}}(\mf g_{i})$ is 
isomorphic to the trivial $\mf g_0$-supermodule, for $i=\pm 1$. 
By Theorem~\ref{Thm::CoindisoInd} and  Corollaries~\ref{Coindsimplesoc} and \ref{SocIndCoind}, 
the mapping \eqref{Eqcorsoc} is well-defined and surjective. 
Also, from Theorem~\ref{Thm::CoindisoInd} and Remark~\ref{InvInKac} it follows
that \eqref{Eqcorsoc} is injective. This proves claim~\eqref{MainThm1.2}.

Now we prove claim~\eqref{MainThm1.1}. By Lemma~\ref{nlem5}, the correspondence
\eqref{Eqcortop} is well-defined. Note that, by Lemma~\ref{Kacsimplesoc}, for any
simple $\mf g_0$-supermodule $X$, the socle of every $K(X)$ is $\mathbb{Z}$-graded.
From claim~\eqref{MainThm1.2} it thus follows that all simple $\mf g$-supermodules
are $\mathbb{Z}$-gradeable. In particular, $L(X)$ is $\mathbb{Z}$-gradeable.
We fix a $\mathbb{Z}$-grading on $L(X)$ such that the top non-zero graded component is
of degree $0$.
For $i\in\mathbb{Z}$, denote by $\langle i\rangle$ the shift of grading functor on $\mathfrak{g}\text{-}\mathrm{smod}^\mathbb{Z}$ which maps homogeneous elements of degree
$j$ to homogeneous elements of degree $j-i$. If $X$ and $Y$ are simple $\mf g_0$-supermodules, 
then  the only chance for $\mathrm{Hom}_{\mathfrak{g}\text{-}\mathrm{smod}^\mathbb{Z}}(L(X),L(Y)\langle i\rangle)$
to be non-zero is when $i=0$ (for any non-zero homomorphism must be an isomorphism and, unless $i=0$, the
top non-zero graded components of $L(X)$ and $L(Y)$ would not match). 
However, as the degree zero component of $L(X)$ is isomorphic to $X$
and the degree zero component of $L(Y)$ is isomorphic to $Y$, in the case $X\not\cong Y$ we have
\begin{displaymath}
 \mathrm{Hom}_{\mathfrak{g}\text{-}\mathrm{smod}^\mathbb{Z}}(L(X),L(Y)\langle i\rangle)=0
\end{displaymath}
since $\mathrm{Hom}_{\mathfrak{g}_0}(X,Y)=0$. This shows that 
the correspondence \eqref{Eqcortop} is injective.

Finally, let $L$ be a simple $\mf g$-supermodule. We can consider it as a 
$\mathbb{Z}$-graded supermodule such that all non-zero components have non-positive degrees
and the degree zero component $X$ is non-zero. Then $X$ is a $\mf g_{\geq 0}$-supermodule.
Due to exactness of $\mathrm{Ind}_{\mf g_{\geq 0}}^{\mf g}$, simplicity of $L$ implies
simplicity of $X$. As ${\mf g_{1}}X=0$, we even get that $X$ is a simple $\mf g_{0}$-supermodule.
By adjunction, there is a non-zero homomorphism from $K(X)$ to $L$. This shows that 
the correspondence \eqref{Eqcortop} is surjective, completing the proof.
\end{proof}

We recall the parity change functor $\Pi$ defined in Section \ref{sectPre} 
and conclude this subsection with the following corollary.

\begin{cor}
We have $L(V)\not \cong \Pi L(V)$, for each simple $\mf g_0$-supermodule $V$.
\end{cor}

\begin{proof} 
Suppose that there is an isomorphism $f: L(V) \cong \Pi L(V)$. 
By Remark \ref{InvInKac} and Corollary \ref{SocIndCoind}
we may note that $L(V)^{\mf g_1} \cong V$,  $\Pi L(V)^{\mf g_1} \cong \Pi V$,
as $\mf g_0$-supermodules. Consequently, we have $f|_V:V \cong \Pi V$, a contradiction.
\end{proof}

  % We have all analogous results for dual Kac module $\text{Ind}_{\mf g_{\leq 0}}^{\mf g}V$. Since Kac modules have simple image of socle of dual Kac modules, this way we give the canonical simple quotient of Kac modules. We may state the main theorem in terms of the quotient of Kac modules as follows:
  
  %\begin{cor} The canonical simple quotient of Kac modules give gise to a bijection between isomorphism classes of simple $\mf g$-supermodules and simple $\mf g_0$-supermodules. \end{cor}
  
%  {\bf Question:} For type II-like Lie superalgebra $\mf g=\bigoplus_{i=-2}^2\mf g_i$,  the restriction $\text{Res}_{\mf g_{\geq 0}}^{\mf g}M$ is not always finitely generated for a given simple $\mf g$-supermodule $M$. Does $M$ come from the socle or the head of Kac modules?

 \subsection{Examples: simple supermodules over classical Lie superalgebra of type I}

%\subsection{Classical Lie superalgebra of type I}

In this subsection, we consider the example of the {\em classical Lie superalgebras of type I} 
(see, e.g., \cite[Chapter 2, 3]{Mu12} and \cite[Section 1.1]{ChWa12}), they are:
\begin{align} 
&(\text{Type \bf A}): ~\mf{gl}(m|n), \mf{sl}(m|n) \text{ and } \mf{sl}(n|n)/\mathbb{C}I_{n|n};\label{cLIa} \\
&(\text{Type \bf C}): ~\mf{osp}(2|2n); \label{cLIc} \\
&(\text{Strange series}): ~\mf{p}(n)\text{ and }  \mf{p}'(n):=[\mf{p}(n), \mf{p}(n)]; \label{cLIp}
\end{align} 
where $m>n \geq 1$ are integers.

Each of these classical Lie superalgebra of type I admits a $\mathbb{Z}_2$-compatible $\mathbb{Z}$-grading 
\begin{align}
&\mathfrak{g}=\mathfrak{g}_{-1} \oplus \mathfrak{g}_0 \oplus
\mathfrak{g}_{+1}. \label{Zgrad}
\end{align}
As an application, we may conclude that Theorem \ref{MainThm1} holds for all classical Lie superalgebra of type I.

\subsubsection{General linear Lie superalgebra $\mf{gl}(m|n)$} \label{s4.2.1}
Let $\C^{m|n}$ be the standard complex superspace of (graded) dimension $(m|n)$. 
With respect to a fixed ordered homogeneous basis in $\C^{m|n}$, the 
 {\em general linear Lie superalgebra} $\g=\gl(m|n) = \mf{gl}(\C^{m|n})$ can be realized as
 the space of $(m+n)\times (m+n)$ matrices over $\C$. The even subalgebra 
 $\mf g_\oo$ of $\g$ is isomorphic to $\mf{gl}(m)\oplus \mf{gl}(n)$. 
 We let $e_{ij}$, $1\le i,j\le m+n$, denote the
 $(i,j)$-th matrix unit.
 The {\em Cartan subalgebra} of $\g$ consisting of all diagonal matrices is denoted by
 $\h=\h_{m|n} =\text{span} \{e_{ii}|1\le i\le m+n\}$. We
 denote by $\{\ve_i|1\le i\le m+n\}$ the basis in
 $\h^*=\h_{m|n}^*$ which is dual to $\{e_{ii}|1\le i\le m+n\}$. Let 
 $\Phi=\{\ve_i-\ve_j|1\le i\neq j\le m+n\}$ be the root system of $\g$ and 
 denote by $\Phi_\oo$ and $\Phi_{\bar 1}$ the set of even roots and the set of odd roots, respectively. 
 We also denote the set of positive roots
 by $\Phi^+:= \{\ve_i-\ve_j|1\le i< j\le m+n\}$ and the set of negative roots by $\Phi^-: = - \Phi^+$.
The Weyl group $W=\mf{S}_m\times\mf{S}_n$ acts on $\mf h^*$ in the obvious way.

 For each root $\alpha \in \Phi$, let $\mf g_{\alpha}$ be the corresponding root space. We have the
 triangular decomposition  \begin{align} \label{Eq::TriDecom} &\g =\n^- \oplus \h \oplus
 \n,\end{align} where $\displaystyle \mf{n}=\bigoplus_{\alpha\in\Phi^+}\g_\alpha$ and
 $\displaystyle \mf{n}^-=\bigoplus_{\alpha\in-\Phi^+}\g_\alpha$. 
 The subalgebra $\mf b:=\mf h \oplus \mf n$ of upper triangular matrices 
 is called the {\em (standard) Borel subalgebra}. Also, we let 
 $\mf b_\oo:= (\mf b\cap \g_\oo)\subset \mf g_\oo$ be the standard Borel subalgebra of $\mf g_\oo$.
 
 The $\mathbb{Z}_2$-compatible $\mathbb{Z}$-grading of $\mathfrak{g}$ 
 in terms of matrix realization are given by letting  
 \begin{align}
 &\mathfrak{g}_{0} =\mf g_\oo= 
 \left\{ \left( \begin{array}{cc} a & 0\\
 0 & d\\
 \end{array} \right) \| ~a\in \C^{m^2}, d\in \C^{n^2} \right\},\label{Zgradg0}\\ 
 &\mathfrak{g}_{+1} =\bigoplus_{\alpha\in \Phi_{\bar 1}^+} \mf g_{\alpha} = \bigoplus_{1\leq i,j\leq n}\C e_{i,m+j} = 
 \left\{ \left( \begin{array}{cc} 0 & b\\
 0 & 0\\
 \end{array} \right) \| ~b\in \C^{mn} \right\},\label{Zgradgl}\\ 
 &\mathfrak{g}_{-1} =\bigoplus_{\alpha\in \Phi_{\bar 1}^-} \mf g_{\alpha} = \bigoplus_{1\leq i,j\leq n} \C e_{m+i,j} =
 \left\{ \left( \begin{array}{cc} 0 & 0\\
 c & 0\\
 \end{array} \right) \| ~c\in \C^{mn} \right\}.\label{Zgradg-l}
 \end{align} 
 In particular, $\mf g_1 \cong \mathbb{C}^m\otimes (\mathbb{C}^{n})^*$ and 
 $\mf g_{-1} \cong \mathbb{C}^{n}\otimes (\mathbb{C}^{m})^*$ as $\mf g_\oo$ modules. 
 We let $$\mf s : = [\mf g_\oo, \mf g_\oo] \cong \mf{sl}(m)\oplus \mf{sl}(n)$$ be 
 the maximal semisimple ideal of $\mf g_\oo$.
 
 We define the associated grading operator $d^{\mf{gl}(m|n)}$ for $\mf{gl}(m|n)$ as follows: 
 \begin{align} \label{Eq::GrOp} &d^{\mf{gl}(m|n)}:= \sum_{m+1\leq i\leq m+n} e_{ii} = 
\left( \begin{array}{cc} 0 &0 \\
 0 & {I}_{m}\\
 \end{array} \right)
  \in \mf z(\mf g_\oo).\end{align}
 If $V$ is a simple $\mf g_0$-supermodule, then $d^{\mf{gl}(m|n)}$ acts on $V$ as a scalar $d^{\mf{gl}(m|n)}_V \in \mathbb{C}$ by Dixmier's theorem, see, e.g., \cite[Proposition 2.6.8]{Di96}. Therefore, $K(V)$ can be decomposed into
 $d$-eigenspaces: \begin{align}\label{Eq::GrOnKV} K(V) =
 \bigoplus_{i=0}^{\text{dim}(\mf g_{-1})} K(V)_{d^{\mf{gl}(m|n)}_V -i} \cong \bigoplus_{i=0}^{\text{dim}(\mf g_{-1})} \Lambda^{i}(\mf g_{-1})\otimes V, \end{align}
 where 
$K(V)_{d^{\mf{gl}(m|n)}_V-i}$ is the eigenspace of $d^{\mf{gl}(m|n)}$  with
eigenvalue $d^{\mf{gl}(m|n)}_V -i$. This means that the homogeneous components of 
the $\mathbb{Z}$-grading on $L(V)$ are eigenspaces for $d^{\mf{gl}(m|n)}$ with different eigenvalues.
 
 For given positive integers $m,n$, the subsuperalgebra $$\mf{sl}(m|n): = [\mf{gl}(m|n), \mf{gl}(m|n)],$$ is called the {\em special linear Lie superalgebra}.  The superalgebra $\mf{sl}(m|n)$ is the kernel of the supertrace on $\mf{gl}(m|n)$. We have $\mf{sl}(m|n)_0 = \mf{sl}(m)\oplus \mf{sl}(n)\oplus \C I_{m|n},$ where $$I_{m|n}: =\left( \begin{array}{cc} nI_m &0 \\
 0 & m{I}_{n}\\
 \end{array} \right),$$
 and $\mf{sl}(m|n)_1 = \mf{gl}(m|n)_1 , \mf{sl}(m|n)_{-1} = \mf{gl}(m|n)_{-1}.$ Note that $\mf{gl}(m|n)_0 = \mf{sl}(m|n)_0\oplus \C d^{\mf{gl}(m|n)}$.

The Lie superalgebra $\mf{sl}(m|n)$ is simple if and only if $m\neq n$, moreover, 
$\mf{sl}(n|n)/\C I_{n,n}$ is a simple Lie superalgebra as well.

\subsubsection{Orthosymplectic Lie superalgebra $\mf{osp}(2|2n)$} 
The orthosymplectic Lie superalgebra $\mf{osp}(m|2n)$ is the subsuperalgebra of $\mf{gl}(m|2n)$  
preserving a non-degenerated supersymmetric bilinear forms 
(see, e.g., \cite[Section 1.1.3]{ChWa12} and \cite[Section 2.3]{Mu12}). 
In particular, $\mf{osp}(2|2n)$ is a classical Lie superalgebra  of type I:
\[ \mf{osp}(2|2n)=
 \left\{ \left( \begin{array}{cccc} c &0 & x &y\\
 0 & -c& v & u\\
 -u^t& -y^t & a &b\\
 v^t &x^t & c& -a^t \\
 \end{array} \right)\|\text{$c\in \C$; $x,y,v,u\in \C^{n}$; $a,b,c\in \C^{n^2}$; $b=b^t$, $c=c^t$} \right\}.
\] 
 The $\mathbb{Z}_2$-compatible $\mathbb{Z}$-grading of $\mf{osp}(2|2n)$ are given as follows:
\begin{align*}
 &\mf{osp}(2|2n)_0=
 \left\{ \left( \begin{array}{cccc} c &0 & 0 &0\\
 0 & -c& 0 & 0\\
 0& 0 & a &b\\
 0 &0 & c& -a^t \\
 \end{array} \right)\|\text{$c\in \C$; $a,b,c\in \C^{n^2}$; $b=b^t$, $c=c^t$} \right\}, \\
 &\mf{osp}(2|2n)_1=
 \left\{ \left( \begin{array}{cccc} 0 &0 & x &y\\
 0 & 0& 0 & 0\\
 0& -y^t & 0 &0\\
 0 &x^t & 0& 0 \\
 \end{array} \right)\|\text{$x,y\in \C^{n}$} \right\}, \\
 &\mf{osp}(2|2n)_{-1}=
 \left\{ \left( \begin{array}{cccc} 0 &0 & 0 &0\\
 0 & 0& v & u\\
 -u^t& 0& 0 &0\\
 v^t & 0& 0& 0 \\
 \end{array} \right)\|\text{$v,u\in \C^{n}$} \right\}.
\end{align*}

We define the associated grading operator $d^{\mf{osp}(2|2n)}$ for $\mf{osp}(2|2n)$ as follows: 
\begin{align} \label{Eq::GrOpforosp} &d^{\mf{osp}(2|2n)}:=\left( \begin{array}{cccc} 1 &0 & 0 &0\\
 0 & -1& 0 & 0\\
 0& 0 & 0 &0\\
 0 &0 & 0& 0 \\
 \end{array} \right)
\in \mf z(\mf{osp}(2|2n)_0).\end{align}
Observe that $\mf{osp}(2|2n) = \C d^{\mf{osp}(2|2n)}\oplus \mf{sp}(2n)$. 
If $V$ is a simple $\mf g_0$-supermodule, then $d^{\mf{osp}(2|2n)}$ acts on $V$ as a scalar $d^{\mf{osp}(2|2n)}_V \in \mathbb{C}$ by Dixmier's theorem (\cite[Proposition 2.6.8]{Di96}). Just like in 
Subsection~\ref{s4.2.1}, the homogeneous components of 
the $\mathbb{Z}$-grading on any $K(V)$ or $L(V)$ are eigenspaces for $d^{\mf{osp}(2|2n)}$ with different eigenvalues.
 
\subsubsection{Periplectic Lie superalgebra $\mf{p}(n)$}
  
The periplectic Lie superalgebra $\mf{p}(n)$ is a subalgebra of $\mf{gl}(n|n)$ preserving a non-degenerated odd symmetric bilinear form (see, e.g., \cite[Section~1.1.5]{ChWa12}).
The standard matrix realization is given
by
\[ \mf{p}(n)=
\left\{ \left( \begin{array}{cc} a & b\\
c & -a^t\\
\end{array} \right)\| ~ a,b,c\in \C^{n^2},~b=b^t\text{ and }c=-c^t \right\}.
\]
The superalgebra $\mf p(n)$ admits a $\mathbb{Z}_2$-compatible $\mathbb{Z}$-grading inherited from the $\mathbb{Z}$-grading \eqref{Zgradg0}, \eqref{Zgradgl} and \eqref{Zgradg-l}.
  Namely, 
  \begin{align*}
  &\mf p(n)_0 = \mf p(n)_\oo = \bigoplus_{1\leq i,j \leq n}\C(e_{ij} -e_{n+j,n+i}) = 
  \left\{ \left( \begin{array}{cc} a & 0\\
  0 & -a^t\\
  \end{array} \right)\| ~a\in \C^{n^2} \right\}. \\
  &\mf p(n)_1 = \bigoplus_{1\leq i\leq j \leq n}\C(e_{i,n+j} +e_{j,n+i}) = 
  \left\{ \left( \begin{array}{cc} 0 & b\\
  0 & 0 \\
  \end{array} \right)\| ~b\in \C^{n^2}, b=b^t \right\}, \\
  &\mf p(n)_{-1} = \bigoplus_{1\leq i<j \leq n}\C(e_{n+i,j}  - e_{n+j,i}) = 
  \left\{ \left( \begin{array}{cc} 0 & 0\\
  c & 0\\
  \end{array} \right)\| ~c\in \C^{n^2}, c=-c^t \right\}.
  \end{align*}
  We may note that $\mf p(n)_{-1} \cong \Lambda^2(\mathbb{C}^{n*})$ and
  $\mf p(n)_{+1} \cong S^{2}(\mathbb{C}^n)$, as
  $\mf p(n)_0$-supermodules. The subalgebra $\mf p(n)' := [\mf p(n), \mf p(n)]$ inherits a  $\mathbb{Z}_2$-compatible $\mathbb{Z}$-grading as follows: 
  $$\mf p(n)'_0 \cong \mf{sl}(n), \quad\mf p(n)'_1 = \mf p(n)_1 \quad\text{ and }\quad\mf p(n)'_{-1} = \mf p(n)_{-1}.$$

We define the associated grading operator $d^{\mf p(n)}$ for $\mf{p}(n)$ as follows: \begin{align} \label{Eq::GrOpforp} &d^{\mf p(n)}:= \sum_{1\leq i\leq n} (e_{ii} - e_{n+i,n+i})= 
\left( \begin{array}{cc} I_n &0 \\
0 & -{I}_{n}\\
\end{array} \right)
\in \mf z(\mf p(n)_0).\end{align} 
Then  $\mf p(n)_0 = \mf p(n)'_0\oplus \C d^{\mf p(n)}$. If $V$ is a simple $\mf g_0$-supermodule, then $d^{\mf{p}(n)}$ acts on $V$ as a scalar $d^{\mf{p}(n)}_V \in \mathbb{C}$ by Dixmier's theorem (\cite[Proposition 2.6.8]{Di96}). 
Just like in 
Subsection~\ref{s4.2.1}, the homogeneous components of 
the $\mathbb{Z}$-grading on any $K(V)$ or $L(V)$ are eigenspaces for $d^{\mf{p}(n)}$ with different eigenvalues.

\subsection{Simple supermodules for classical Lie superalgebras of type I}  

The following statement is an immediate consequence from the properties of the grading operator
mentioned above:

\begin{cor}  
Let $\mf g = \mf{gl}(m|n), \mf{osp}(2|2n)$ or $\mf p(n)$.
Let $V$ be a simple $\mf g_0$-supermodule. For any submodule $N \subseteq K(V)$, 
the  decomposition 
$$N = \bigoplus_{k\geq 0}\left(N\cap (\Lambda^k\mf g_{-1} \otimes V)\right)$$ 
is the eigenspace decomposition
with respect to the action of $d^{\mf g}$. In particular, if we consider 
the standard $\mathbb{Z}$-grading on $K(V)$, then all $\mf g$-submodules of
$K(V)$ are, automatically, $\mathbb{Z}$-graded submodules.
\end{cor}

The following proposition reduces the study of simple supermodules over all
classical Lie superalgebras of type I, see  \eqref{cLIa}, to the study of 
simple supermodules over $\mf{gl}(m|n)$, $\mf{osp}(2|2n)$ and $\mf p(n)$.
 
\begin{prop}\label{nprop7}
Let $\mf g= \mf{gl}(m|n), \mf p(n)$, and  set 
$\mf g':= [\mf g,\mf g]  \emph{(}=\mf{sl}(m|n), \mf p(n)'\emph{)}$. 
\begin{enumerate}[$($i$)$]
\item\label{nprop7.1} If $M$ is a simple $\mf g$-supermodule, then $\emph{Res}_{\mf g'}^{\mf g}(M)$ 
is a simple $\mf g'$-supermodule.
\item\label{nprop7.2} If $M$ and $N$ are simple $\mf g$-supermodules, 
then $\emph{Res}_{\mf g'}^{\mf g}(M)\cong \emph{Res}_{\mf g'}^{\mf g}(N)$
if and only if $M^{\mf g_1}$ and $N^{\mf g_1}$ are isomorphic as $\mf g'_0$-supermodules.
\item\label{nprop7.3} Every simple $\mf g'$-supermodule has the form 
$\emph{Res}_{\mf g'}^{\mf g}(M)$, for some  $M$ as above.
\end{enumerate}
\end{prop}

Note that the difference between $\mf g_0$ and $\mf g'_0$ is given by the central element
$d^{\mf g}$. Therefore Proposition~\ref{nprop7}\eqref{nprop7.2} says that the elements in each fiber of the map
$M\mapsto \mathrm{Res}_{\mf g'}^{\mf g}(M)$ are indexed by complex numbers which prescribe the
scalar with which $d^{\mf g}$ acts on the simple $\mf g_0$-supermodule $M^{\mf g_1}$.

\begin{proof}
As the difference between $\mf g_0$ and $\mf g'_0$ is given by the central element
$d^{\mf g}$, the restriction of a simple $\mf g_0$-supermodule to $\mf g'_0$ remains simple,
moreover, this restriction map is surjective.

Let $M$ be a simple $\mf g$-supermodule. From Corollary~\ref{SocIndCoind} we have that
there is a simple $\mf g_0$-supermodule $V$ such that
$M\cong \mathrm{Soc}(\mathrm{Ind}_{\mathfrak{g}_{\geq 0}}^{\mathfrak{g}}(V)) = 
U(\mf g)\cdot(\Lambda^{\text{max}}(\mf g_{-1})\otimes V)$.
As $d^{\mf g}$ acts as a scalar on the simple $\mf g_0$-supermodule
$\Lambda^{\text{max}}(\mf g_{-1})\otimes V$, we have 
$$U(\mf g)\cdot(\Lambda^{\text{max}}(\mf g_{-1})\otimes V)=
U(\mf g')\cdot(\Lambda^{\text{max}}(\mf g_{-1})\otimes V).$$
As $\Lambda^{\text{max}}(\mf g_{-1})\otimes V$ is a simple $\mf g_0$-supermodule, 
it follows that $\text{Res}^{\mf g_0}_{\mf g'_0}\big(\Lambda^{\text{max}}(\mf g_{-1})\otimes V\big)$ 
is also simple. Therefore we have
$$\text{Res}_{\mf g'}^{\mf g}(M) \cong \text{Res}_{\mf g'}^{\mf g}\big(U(\mf g')\cdot(\Lambda^{\text{max}}
(\mf g_{-1})\otimes V)\big) = 
\text{Soc}\big(\text{Ind}^{\mf g'}_{\mf g'_{\geq 0}}(\text{Res}_{\mf g'_0}^{\mf g_0}(V))\big),$$ 
which is a simple $\mf g'$-supermodule by Lemma \ref{Kacsimplesoc}. This proves claim~\eqref{nprop7.1}.
 
Let now $M'$ be a simple $\mf g'$-supermodule. Then, by Corollary \ref{SocIndCoind}, 
there is a simple $\mf g'_0$-supermodule $V'$ such that
$$M'\cong \text{Soc}\big(\text{Ind}_{\mf g'_{\geq 0}}^{\mf g'}(V')\big) = 
U(\mf g')\cdot(\Lambda^{\text{max}}(\mf g_{-1})\otimes V').$$
Let $V$ be any simple $\mf g_0$-supermodule such that $\text{Res}_{\mf g'_0}^{\mf g_0}(V) = V'$, 
then we have
$$\text{Res}_{\mf g'}^{\mf g}\big(\text{Soc}(\text{Ind}_{\mf g_{\geq 0}}^{\mf g}(V))\big)= U(\mf g')\cdot(\Lambda^{\text{max}}(\mf g_{-1})\otimes V') =\text{Soc}(\text{Ind}_{\mf g'_{\geq 0}}^{\mf g'}(V'))\cong M'.$$ 
Claims~\eqref{nprop7.2} and \eqref{nprop7.3} follow. 
\end{proof}

\subsection{Criteria for simplicity of Kac modules} \label{Sect::CriforSimpleKac} 

In the beginning parts of this subsection, we assume that $\mf g$ 
is one of the classical Lie superalgebra of type I with 
$\text{dim}\mf g_1=\text{dim}\mf g_{-1}$, namely,  
$$\mf g = \mf{gl}(m|n),\quad \mf{sl}(m|n) 
,\quad  \mf{sl}(n|n)/\mathbb{C}I_{n|n}\quad  \text{ or } \quad  \mf{osp}(2|2n).$$
The periplectic Lie superalgebra $\mf p(n)$ will be discussed in Subsection~\ref{Criforpn}.
A criterion for simplicity of finite dimensional Kac module was given by Kac in \cite{Ka78} 
in terms of typicality of highest weights. In this section, we provide criteria for 
simplicity of Kac modules for arbitrary (simple) input of Kac functor.

\subsubsection{BGG category of $\mf g$-supermodules and Duflo's theorem}
Following \cite{BGG76},  consider the BGG category
$\mathcal{O}=\mathcal{O}(\mathfrak{g},\mathfrak{h},\mathfrak{n})$ associated to
the standard triangular decomposition 
\begin{displaymath}
\mathfrak{g}=\mathfrak{n}^-\oplus \mathfrak{h}\oplus \mathfrak{n} 
\end{displaymath}
of $\mathfrak{g}$. 
It is the full subcategory of $\mf g$-smod consisting 
of all $\mf g$-supermodules on which $\mf h$ acts semisimply and $\mf b$ acts locally finitely.
We set $\mc O_\oo : = \mc O(\mf g_\oo,\mf h_\oo,\mf n_\oo)$. 
For $\la \in \mf h^*$, we denote by $V(\la)$ the simple even $\mf b_\oo$-highest weight 
$\mf g_\oo$-supermodule with highest weight $\la$. We set $K(\la):=K(V(\lambda))$.

For $\la \in \mf h^*$, the corresponding {\em Verma supermodule} $\Delta(\la)$ (over $\mf g$) is defined by
$${\Delta}(\la):=U(\mf{g})\otimes_{\mf b} \mathbb C_{\la},$$
where $ \mathbb C_{\la}$ is the even one-dimensional $\mf b$-supermodule module corresponding to $\lambda$.
The unique simple quotient of $\Delta(\la)$ is denoted by $L(\lambda)$. We let $\chi_\la$ 
(resp. $\chi_{\la}^\oo$) be the $U$-central (resp. $U_\oo$-central) character corresponding to $\la$. 

Denote by $\rho\in \mf h^*$ the Weyl vector as in \cite[Remark 1.21]{ChWa12}: 
\begin{align}\label{Eq::WelyVc}
&\rho = \frac{1}{2} \sum_{\alpha \in \Phi_{\oo}}\alpha - \frac{1}{2}\sum_{\alpha \in \Phi_{\one}}\alpha.
\end{align} 
Let $(\cdot,\cdot)$ be the non-degenerated $W$-invariant form on~$\mf h^*$ as defined in \cite[Section 1.2]{ChWa12}. We consider the dot-actions of $W$ given by  $w\cdot\la=w(\la+\rho)-\rho$, for all $w\in W$ and~$\lambda\in \mf h^*$. A weight $\lambda$ is called {\em integral} if $(\lambda,\alpha) \in\mathbb{Z}$ for all even roots $\alpha$. An integral weight  is called {\em dominant} if $\la$ is dominant for the dot-action of $W$ and {\em regular} if $( \la+\rho,\alpha)\neq 0$, for all simple even roots $\alpha$.  A weight $\la$ is called {\em typical} if $(\lambda+\rho,\alpha)\neq 0$, for all odd roots $\alpha$. 

\begin{thm}\emph{(}cf. \cite{Du77}\emph{)} \label{Dufthm}
Let $V$ be a simple $\mf g_0$-supermodule. Then there exist $\la \in \mf h$ 
such that $\emph{Ann}_{U_\oo}(V) = \emph{Ann}_{U_\oo}(V(\la))$.
\end{thm}

For a given $\mf g$-supermodule (resp. $\mf g_\oo$-supermodule) $X$, we denote its $U$-annihilator 
(resp. $U_\oo$-annihilator) by $\text{Ann}_U(X)$ (resp.  $\text{Ann}_{U_\oo}(X)$).
Together with Theorem~\ref{Dufthm}, the following lemma shows that annihilators of 
arbitrary Kac modules are annihilators of Kac modules in $\mc O$. 

\begin{lem} \label{AnnKac} 
Let $V$ and $W$ be simple $\mf g_\oo$-supermodules such that 
$\emph{Ann}_{U_\oo}(V) = \emph{Ann}_{U_\oo}(W)$. Then 
$\emph{Ann}_U (K(V)) =\emph{Ann}_U (K(W))$.
\end{lem}

\begin{proof}
Mutatis mutandis  the proof of \cite[Proposition 5.1.7]{Di96}.
\end{proof}

\subsubsection{Simplicity criteria} 
Fix two non-zero elements 
$$X^{-}\in \Lambda^{\text{max}}(\mf g_{-1})\quad\text{ and }\quad X^{+}\in \Lambda^{\text{max}}(\mf g_{1}).$$ 
  
\begin{lem} 
We have the decomposition 
\begin{align} \label{Eq::Omega} &X^{+} X^- = \Omega + \sum_i x_i r_i y_i,\end{align}
for some $x_i \in \Lambda(\mf g_{ -1})\backslash \mathbb{C}$, 
$y_i \in \Lambda(\mf g_{1})\backslash \mathbb{C}$, 
$r_i\in U(\mf g_{\oo})$ and $\Omega \in Z(\mf g_\oo)$. 
\end{lem} 

\begin{proof}
As $X^{+}\cdot X^-$ has $\mathfrak{h}$-weight $0$, by the Poincar{\'e}-Birkhoff-Witt Theorem, 
we have the decomposition \eqref{Eq::Omega} such that  
$$x_i \in \Lambda(\mf g_{ -1})\backslash \mathbb{C},\quad 
y_i \in \Lambda(\mf g_{1})\backslash \mathbb{C}\quad\text{ and }\quad r_i, \Omega \in U(\mf g_0).$$
It remains to show that $\Omega \in Z( \mf g_\oo)$. 

For $r\in \mf g_\oo$, we want to show that $r\Omega - \Omega r =0$. 
The one-dimensional $\mf g_\oo$-representation 
$\Lambda^{\text{max}}(\mf g_{1})\otimes \Lambda^{\text{max}}(\mf g_{-1})$ 
is concentrated in the $\mathfrak{h}$-weight $0$. Hence, 
$rX^+X^- = X^+X^-r$. Since $$r( \sum_i x_i r_i y_i) - ( \sum_i x_i r_i y_i)r =  \sum_i x'_i r'_i y'_i,$$ 
for some $r'_i\in U(\mf g_0)$, $x_i' \in \Lambda(\mf g_{ -1})\backslash \mathbb{C}$ and 
$y_i' \in \Lambda(\mf g_{1})\backslash \mathbb{C}$, 
the claim follows from the Poincar{\'e}-Birkhoff-Witt Theorem. 
\end{proof}

\begin{example}
Let $\mf g:=\mf{gl}(m|1)$ and choose $X^{\pm}$ as follows:
$$X^+:= e_{1,m+1}e_{2,m+1}\cdots e_{m,m+1},\qquad
X^-:= e_{m+1,1}e_{m+1,2}\cdots e_{m+1,m}.$$
By a direct calculation, we have the expression 
$$\Omega = \sum_{\sigma \in \mf S_m}(-1)^{\sigma} X_{m,\sigma(m)}X_{m-1,\sigma(m-1)}\cdots X_{1,\sigma(1)},$$
where $X_{ij}:=[e_{i,m+1},e_{m+1,j}]+\delta_{ij}(m-i)1_{U}\in U(\mf g_0)$, for  $1\leq i,j\leq m$.
\end{example}

Now we are ready to formulate our first simplicity criterion for $K(V)$.

\begin{thm} \label{thm::criforsimpleKac} 
Let $V$ be a simple $\mf g_\oo$-supermodule. Then the following assertions are equivalent.
\begin{enumerate}[$($a$)$]
\item\label{thmkac.1} $K(V)$ is simple.
\item\label{thmkac.2} $\Omega$ acts on $V$ as a injective linear operator.
\item\label{thmkac.3} $\Omega$ acts on $V$ as a non-zero scalar.
\end{enumerate}
\end{thm} 
  
\begin{proof}
We first prove \eqref{thmkac.1}$\Rightarrow$\eqref{thmkac.2}. Suppose that $K(V)$ is simple but $\Omega v =0$, for some non-zero $v\in V$. On the one hand, simplicity of $K(V)$ and Lemma~\ref{Kacsimplesoc}
imply $$U_\oo \Lambda(\mf g_1)\Lambda (\mf g_{-1})\cdot (\Lambda^{\text{max}} (\mf g_{-1}) \otimes v) =  K(V)$$ 
and hence $U_\oo\Lambda^{\text{max}}(\mf g_1)\cdot (\Lambda^{\text{max}} (\mf g_{-1}) \otimes v) = V$. On the other hand, $\Omega v =0$ means that $\Lambda^{\text{max}}(\mf g_1)\cdot (\Lambda^{\text{max}}(\mf g_{-1}) \otimes v) =0$, a contradiction. 
  	
We next prove \eqref{thmkac.2}$\Rightarrow$\eqref{thmkac.1}. In this case we have $X^+ X^- (1_U\otimes v) = 1_U\otimes \Omega v\neq 0$ in $V\subset K(V)$, for any non-zero $v \in V$. From Lemma \ref{Kacsimplesoc} it thus follows that  $\text{soc}(K(V)) \supseteq  K(V)$, that is, $K(V)$ is simple.

The equivalence  of \eqref{thmkac.2} and \eqref{thmkac.3} follows from Dixmier's theorem 
\cite[Proposition~2.6.8]{Di96}. 
\end{proof}
  
We now give our second criterion for simplicity of Kac module which is formulated in terms of $U_\oo$-annihilators. 
For a given $\lambda \in \mf h^*$, recall that $V(\lambda)$ denotes the simple highest weight $\mf g_\oo$-supermodule of highest weight $\lambda$ with respect to the Borel subalgebra $\mf b_\oo$.  The following corollary shows that simplicity of Kac modules can be  determined in terms of the annihilator of the simple $\mf g_0$-input of Kac functor.

\begin{cor}\label{aaa1}
Let $V$ and $W$ be two simple $\mf g_\oo$-supermodules. If $\emph{Ann}_{U_\oo}(V) = \emph{Ann}_{U_\oo}(W)$, then the following assertions are equivalent:
\begin{enumerate}[$($a$)$]
\item \label{aaa1.1} $K(V)$ is simple.
\item \label{aaa1.2} $K(W)$ is simple.
\item \label{aaa1.3} $\emph{Ann}_{U_\oo}(V) = \emph{Ann}_{U_\oo}(V(\la))$, for some typical $\la$.
\end{enumerate}
\end{cor}
  
\begin{proof}
By Theorem~\ref{Dufthm}, there exists $\la \in \mf h^*$ such that $$\text{Ann}_{U_\oo}(V) = \text{Ann}_{U_\oo}(W) = \text{Ann}_{U_\oo}(V(\la)).$$
In particular, all these annihilators contain $U(\mf g_\oo)\chi_{\la}^{\oo}$. Therefore $\Omega$ acts as the same scalar $\chi_{\la}^{\oo}(\Omega)$ on both $V$ and $W$. Therefore \eqref{aaa1.1} and \eqref{aaa1.2} are equivalent by Theorem \ref{thm::criforsimpleKac}. 
   
The fact that $\chi_{\la}^{\oo}(\Omega)\neq 0$ if and only if $\la$ is typical follows
from \cite[Subsection~4.2]{Go} which says that the evaluation at  $\la$ of 
the Harish-Chandra projection of $\Omega$ has the form
\begin{displaymath}
\prod_{\alpha\in\Phi^+_{\bar{1}}}(\la+\rho,\alpha).
\end{displaymath}
This completes the proof.
\end{proof}

It is natural to consider also the simplicity problem for the opposite Kac module $K'(V)$ of $V$. 
Recall that $2\rho_{\one}$ denotes the sum of all odd positive roots. If $\la$ is such that $V= V(\la)$ is finite-dimensional, then it is well-known that 
the simplicity of $K(\la)$ is equivalent to the simplicity 
of $K'(\la - 2\rho_{\one})$, which is equivalent to the typicality of $\la$,
see e.g. \cite[Lemma 3.3.1]{Ge98} and \cite{Ka78}. The highest weight 
$\mf g_\oo$-supermodule $V(2\rho_{\one})$ is one-dimensional and hence can be denoted by
$\mathbb{C}_{2\rho_{\one}}$. We now extend the comparison of simplicity of Kac modules and
opposite Kac modules to full generality.

\begin{cor} \label{cor::criforsimpleoppoKac}
Let $V$ be a simple $\mf g_\oo$-supermodules.  Then the following are equivalent:
\begin{enumerate}[$($a$)$]
\item \label{Cor::Cri2::KVsimple} $K'(V)$ is simple.
\item \label{Cor::Cri2::KWsimple} $K(V\otimes\mathbb{C}_{2\rho_{\one}})$ is simple.
\item \label{cor::crif3} $\mathrm{Ann}_{U_\oo}(V)=\mathrm{Ann}_{U_\oo}(V(\la-2\rho_{\one}))$, for some typical $\la$.
\end{enumerate}
If any of the above conditions is satisfied, then $K'(V) \cong K(V\otimes\mathbb{C}_{2\rho_{\one}})$.
\end{cor}

\begin{proof} 	
An analogue of the decomposition  \eqref{Eq::GrOnKV} for $K'(V)$ yields existence of a non-zero homomorphism 
$\varphi:K(V\otimes\mathbb{C}_{2\rho_{\one}})\to K'(V)$ whose image coincides with the socle of $K'(V)$ by
Lemma~\ref{Kacsimplesoc}. A similar argument gives a non-zero homomorphism 
$\psi:K'(V)\to K(V\otimes\mathbb{C}_{2\rho_{\one}})$ whose image coincides with the socle of
$K(V\otimes\mathbb{C}_{2\rho_{\one}})$. If any of $K'(V)$ or $K(V\otimes\mathbb{C}_{2\rho_{\one}})$ is simple, then
both $\varphi\circ\psi$ and $\psi\circ\varphi$ are isomorphisms when restricted to the eigenspaces of 
both extremal eigenvalues
of $d^{\mf g}$. As $K'(V)$ is generated by one of these extremal eigenspaces and $K(V\otimes\mathbb{C}_{2\rho_{\one}})$
is generated by the other one, we obtain that both $\varphi$ and $\psi$ are isomorphism. 
The equivalence between \eqref{Cor::Cri2::KVsimple} and \eqref{Cor::Cri2::KWsimple} follows.
The equivalence between \eqref{cor::crif3} and \eqref{Cor::Cri2::KWsimple} follows from 
Corollary~\ref{aaa1}.
\end{proof}

\subsubsection{} \label{Criforpn} 
Here we discuss the periplectic Lie superalgebra $\mf p(n)$ which has 
been excluded in the previous parts of this subsections. For $\mf p(n)$ we can also define 
the Cartan subalgebra $\mf h_{\mf p(n)}:= \mf p(n)\cap \h$, the Borel subalgebra 
$\mf b_{\mf p(n)} = \mf b_{\mf p(n)_0} \oplus \mf p(n)_1$ and the corresponding BGG categories $\mc O$ 
and $\mc O_\oo $.

For  $\mf p(n)$, the principal difficulty is the asymmetry of negative and 
positive roots. However, in \cite[Corollary 5.8]{Se02},
it is shown that, for a simple $\mf p(n)_0$-supermodule 
$V$, the corresponding Kac module $K(V)$ is simple if $V$ admits a typical central character, 
which is a $\mf p(n)$-analog of our Corollary~\ref{aaa1}. 
This asymmetry of positive and negative roots makes the opposite Kac modules 
always non-simple. It also enables us to construct indecomposable modules from 
the difference between Kac and opposite Kac modules. 
	
\begin{prop} \label{OpKpnprop}
Let $V$ be a simple $\mf p(n)_0$-supermodule. Then $K'(V)$ is indecomposable but not simple.
\end{prop}

\begin{proof} Let us denote $\mf p(n)$ by $\mf g$ in this proof. 
Consider $\mathfrak{g}\text{-}\mathrm{mod}^{\mathbb{Z}}$. 
By the universal property of induced modules, 
in $\mathfrak{g}\text{-}\mathrm{mod}^{\mathbb{Z}}$ we have a non-zero
homomorphism $f$ from $K(\Lambda^{\text{max}} (\mf g_1) \otimes V)$
to $K'(V)\langle\dim(\mf g_1)\rangle$. Note that the minimal non-zero homogeneous component
of $K(\Lambda^{\text{max}} (\mf g_1) \otimes V)$ has degree $-\dim(\mf g_{-1})$
while the minimal non-zero homogeneous component
of $K'(V)\langle\dim(\mf g_1)\rangle$ has degree $-\dim(\mf g_{1})$ which is strictly smaller than 
$-\dim(\mf g_{-1})$.
Therefore $f$ cannot be surjective. This implies that $K'(V)$ is not simple.

The fact that $K'(V)$ is indecomposable is proved in Lemma~\ref{Kacsimplesoc}.
\end{proof}

\section{Rough structure of Kac modules}\label{s5}

\subsection{Coker-categories}\label{s5.1}

In this section, we assume that $\mf g=\mathfrak{gl}(m|n)$. For a $\mf g$-supermodule $P$, we denote 
by $\mc C_P$ the {\em coker}-category of $P$, that is $\mc C_P$ is the full subcategory of the 
category of all $\mf g$-supermodules, which consists of all modules $M$ which have a presentation 
$$X\to Y \twoheadrightarrow M,$$ where $X$ and $Y$ are isomorphic to direct summands of 
$P\otimes E$, for some finite dimensional weight $\mf g$-supermodule $E$. 
Similarly, we define {\em coker}-categories for modules over Lie algebras, see e.g. 
\cite{MaSt08}.
  
In this section we describe a part of the structure of Kac modules with arbitrary simple input,
called the {\em rough structure} in \cite{MaSt08} by comparing it with the rough structure
of Kac modules in BGG category  $\mc O$.

\subsection{Harish-Chandra bimodules}\label{Subsection::HCB}
In this subsection we collect all necessary preliminaries about 
the main technical ingredient in the study of rough structure, 
namely, about Harish-Chandra bimodules. 
 
\subsubsection{} Here we introduce Harish-Chandra bimodules.  Let us start with $U_\oo$. 
The full subcategory of $\mf g_\oo$-mod which consists of all finite-dimensional 
weight modules is denoted by $\mc F_\oo$.
Each $U_\oo$-$U_\oo$-bimodule $M$ can be considered as a $\mf g_\oo$-module $M^{\ad}$ with respect
to the adjoint action of $\mf g_\oo$.
The category $\mc H_\oo$ of {\em Harish-Chandra $U_\oo$-$U_\oo$-bimodules} is defined as the
full subcategory in the category of all finitely generated $U_\oo$-$U_\oo$-bimodules which consists of
all bimodules $M$ such that the $\mf g_\oo$-module $M^{\ad}$  is a direct sum of simples in 
$\mc F_\oo$, moreover, each simple appears in $M^{\ad}$ with a finite multiplicity.
For two $\mf g_\oo$-supermodules $M$ and $N$, we denote by $\mathcal{L}(M,N)$ the  $U_\oo$-$U_\oo$-bimodule
of all linear maps from  $M$ to $N$ which are locally finite with respect to the adjoint action of 
$\mf g_\oo$.

The category $\mc H$ of {\em Harish-Chandra $U$-$U$-bimodules} is the full subcategory of the category of 
$U$-$U$-bimodules which consists of all bimodules $M$ whose restriction to 
$U_\oo$-$U_\oo$-bimodules is in $\mc H_\oo$, see \cite[Section 5.1]{MaMe12}. 
Abusing notation, for two $\mf g$-supermodules $M$ and $N$, we denote by $\mathcal{L}(M,N)$ the  $U$-$U$-bimodule
$\mathcal{L}(\text{Res}_{\mf g_\oo}^{\mf g}(M),\text{Res}_{\mf g_\oo}^{\mf g}(N))$.

The full subcategory of $\mf g$-smod which consists of all finite-dimensional 
weight supermodules is denoted by $\mc F$.
For $E \in \mc F$, we define a $\mf g$-bimodule structure on $E\otimes U$ as in \cite[Section 2.2]{BeGe} and \cite[Section 2.4]{Co16}: 
\[ X(v\otimes u)Y = (Xv)\otimes (uY) +(-1)^{\ov X\cdot \ov v}v\otimes (XuY),\]
for all homogeneous $X,Y\in \mf g$, $v\in E$ and $u \in U$.
The following identity is proved in \cite[Section 2.2]{BeGe} in the setup of Lie algebras, however,
the same proof works also for Lie superalgebras:
\begin{align}\label{Eq::1stBeGeIdentity}
\text{Hom}_{U\text{-mod-}U}(E\otimes U, M)\cong \text{Hom}_{U}(E, M^{\ad}). 
\end{align}

\subsubsection{} Let $M$ be a $\mf g$-supermodule, then the $\mf g$-action on $M$ defines a $U$-$U$-homomorphism from $U$ to $\mc L(M,M)$. The kernel of this homomorphism is  $\text{Ann}_U(M)$ and  we have the following embedding of 
$U$-$U$-bimodules: \[U/\text{Ann}_U(M) \hookrightarrow \mc L(M,M).\] 
One says that {\em Kostant's problem} for $M$ has a positive solution if the above embedding is 
an isomorphism, see \cite{Jo80,Go02,MaMe12}. By \cite[Proposition 9.4]{Go02}, which can be applied
as we assumed $\mf g=\mathfrak{gl}(m|n)$, 
Kostant's problem has a positive solution for all typical Verma modules.
We note that  \cite[Proposition 9.4]{Go02} is formulated for {\em strongly typical} Verma modules,
however, for $\mf g=\mathfrak{gl}(m|n)$ the notions of ``typical'' and ``strongly typical''
coincide,  see \cite[Subsection~2.5.5]{Go02}.
 
\subsection{Coker categories for Kac modules} \label{Sect::CateOfKacandEqiv} 
This subsection  generalizes \cite[Section 11.6]{MaSt08}. 
Following \cite[Remark~76]{MaSt08}, for simplicity, we will work with regular integral central characters.
The general case follows from the integral and regular one by standard techniques, in particular using translations
out and on the walls and the equivalences from \cite{ChMaWa13}.

Recall that ${\mf s}=[{\mf g_\oo},{\mf g_\oo}]$. Let $V$ be a simple $\mf g_\oo$-supermodule such that 
$L:= \text{Res}_{\mf s}^{\mf g_\oo}(V)$ admits a regular and integral central character. 
Observe that every simple $\mf g_\oo$-supermodule $S$ is determined uniquely by the 
underlying simple $\mf s$-supermodule $\text{Res}_{\mf s}^{\mf g_\oo}(S)$ and a linear 
functional (depending on $S$) on $\mf z(\mf g_\oo)$. Abusing notation, we use $\cdot$ to denote the 
$W$-action for $\mf s$, that is, $$w\cdot\la=w(\la+\rho_{\oo})-\rho_{\oo},$$ for 
all $w\in W$ and $\la\in \mf h^*_{\mf s}$. By Theorem~\ref{Dufthm},  there is 
a dominant weight $\nu$ and $\sigma\in W$ such that 
$\text{Ann}_{U(\mf s)} (L) = \text{Ann}_{U(\mf s)}V(\sigma\cdot \nu)$. 
We may assume that $\sigma$ is contained in a right cell associated with a parabolic 
subalgebra $\mf p \subseteq \mf s$ as in \cite[Remark 14]{MaSt08}. Therefore there 
is a dominant weight $\mu$ such that the parabolic block $\mc O^{\mf p}_{\mu}$ 
contains exactly one simple module $V(y\cdot \mu)$, and this module is projective 
(see, e.g., \cite[3.1]{IrSh88}). Tensoring, if necessary, with finite dimensional modules,
without loss of generality we may assume that 
$\mu$ is typical and {\em generic} in the sense of \cite[Subsection~5.3]{MaMe12}.  
Let $F$ be the projective functor given in 
\cite[Proposition 61]{MaSt08} and define $\ov N$ to be the simple quotient of
$FL$ (in fact, as it turns out, $\ov N=FL$). We refer the reader to 
\cite[Section~11]{MaSt08} for more details of our setup. In particular, we have that 
$$\text{Ann}_{U(\mf s)}(\ov N) = \text{Ann}_{U(\mf s)}( V(y\cdot \mu))= \text{Ann}_{U(\mf s)}( V(\mu))$$
and, consequently, $\mathrm{Ann}_U(K(\mu))=\mathrm{Ann}_U(K(y\cdot \mu))$, see Lemma~\ref{AnnKac}.

\begin{thm} \label{thm::g0CokerEquiv}\emph{(}\cite[Theorem 66]{MaSt08}\emph{)} 
The functor
$$\Xi_\oo: = \mc L(\ov N, -)\otimes_{U(\mf s)} V(y\cdot \mu): \mc C_{\ov N} \rightarrow \mc C_{V(y\cdot \mu)},$$ 
is an equivalence.
\end{thm}
  
We extend the categories $\mc C_{\ov N}$ and $\mc C_{V(y\cdot \mu)}$ of $\mf s$-supermodules to categories of $\mf g_\oo$-supermodules by allowing arbitrary scalar actions of $\mf z(\mf g_\oo)$. It is proved in \cite[Lemma 67]{MaSt08} that $\mc C_{\ov N}$ and $\mc C_{V(y\cdot \mu)}$ are both admissible in the sense of \cite[Section 6.3]{MaSt08}. 
  
Let $I:= \text{Ann}_U(K(\ov N))=\text{Ann}_U(K(y\cdot \mu))=\text{Ann}_U(K(\mu))$, see Lemma~\ref{AnnKac}. 
Denote by $\mc H_I^1$ the full subcategory of $\mc H$ which 
consists of all bimodules annihilated by $I$ from the right. 
Now we can formulate the following equivalence of coker-categories.

\begin{thm} \label{Thm::Equi}
Assume that the weight $\mu$ defined above is typical. Then there are equivalences of categories,
\begin{align} \label{thm::coker}
\mc C_{K(\overline{N})}\cong \mc H_I^1 \cong   \mc C_{K(y\cdot \mu)},
\end{align}
sending $K(\ov N)$ to $K(y\cdot \mu)$.
\end{thm}

\begin{proof}
Our proof follows \cite[Theorem 5.9]{BeGe}, \cite[Theorem~5]{KhMa04} and \cite[Theorem~5.1]{MaMe12}. 
Note that the second equivalence in \eqref{thm::coker} is just a special case of the first one.
So, we just need to prove the first equivalence.

\begin{lem}\label{nlem11}
Kostant's problem for both $K(\mu)$ and $K(y\cdot \mu)$ has positive solutions. 
\end{lem}

\begin{proof}
For $K(\mu)$, the claim follows from \cite[Proposition~9.4]{Go02} and \cite[6.9 (10)]{Ja83}. 
For a given simple $\mf g$-supermodule $E\in \mc F$, we have 
\[ \dim\mathrm{Hom}_{\mf g}(K(\mu)\otimes E, K(\mu)) = \dim\mathrm{Hom}_{\mf g}(K(y\cdot \mu)\otimes E, K(y\cdot \mu)),\]
by a similar argument used in the proof of \cite[Lemma 70]{MaSt08} and \cite[Theorem 60]{MaSt08}. Hence 
\[ \dim\mathrm{Hom}_{\mf g}(E, \mc L(K(\mu), K(\mu))) = \dim\mathrm{Hom}_{\mf g}(E, \mc L(K(y\cdot \mu), K(y\cdot \mu))),\]
by \cite[6.8(3)]{Ja83}. Since Kostant's problem has a positive solution for $K(\mu)$, it follows that
\[ \dim\mathrm{Hom}_{\mf g}(E, U/I) = \dim\mathrm{Hom}_{\mf g}(E, \mc L(K(y\cdot \mu), K(y\cdot \mu))).\]
As $K(\mu)$ and $K(y\cdot \mu)$ have the same annihilators, it follows that
Kostant's problem has a positive solution for $K(y\cdot \mu)$.
\end{proof}

\begin{lem}\label{nlem15}
Kostant's problem for $K(\ov N)$ has a positive solution. 
\end{lem}

\begin{proof}
To see this, for any simple $\mf g$-supermodule $E\in \mc F$, we have 
\begin{align*}
&\dim\mathrm{Hom}_{\mf g} (E, \mc L(K(\ov N), K(\ov N)))=
\dim\mathrm{Hom}_{\mf g}(K(\ov N)\otimes E, K(\ov N))=
\dim\mathrm{Hom}_{\mf s}(\ov N\otimes E', \ov N),
\end{align*}
for some finite-dimensional $\mf s$-supermodule $E'$. By \cite[Theorem 60 (iii)]{MaSt08} and \cite[Proposition 65]{MaSt08}, Kostant's problem has positive solutions for $\ov N$ and $V(y\cdot \mu)$. Therefore 
\begin{align*}
&\dim\mathrm{Hom}_{\mf s}(\ov N\otimes E', \ov N)  \\
&=[\mc L(\ov N, \ov N): E'] \hskip 3cm (\text{by \cite[6.8(3)]{Ja83}})\\
&= [U(\mf s)/\text{Ann}_{U(\mf s)}(\ov N): E']  \hskip 1.5cm (\text{by \cite[Proposition 65]{MaSt08}})\\
& = [U(\mf s)/\text{Ann}_{U(\mf s)}( V(y\cdot \mu)): E']  \hskip 0.5cm (\text{by \cite[Proposition 65]{MaSt08}}) \\ 
&=  [\mc L( V(y\cdot \mu), V(y\cdot \mu)): E'] \hskip 1.1cm (\text{by \cite[Theorem 60 (iii)]{MaSt08}})\\
&= \dim\mathrm{Hom}_{\mf s}(V(y\cdot \mu)\otimes E',  V(y\cdot \mu)).
\end{align*}
Consequently, we obtain 
\begin{align}
\dim\mathrm{Hom}_{\mf g} (E, \mc L(K(\ov N), K(\ov N))) = \dim\mathrm{Hom}_{\mf g}(E, \mc L(K( y\cdot \mu), K( y\cdot \mu))).
\end{align}
Since Kostant's problem has a positive solution for $K(y\cdot \mu)$ by Lemma~\ref{nlem11}
and $K(\ov N)$ and $K(y\cdot \mu)$ have the same annihilators, we have
\[ \dim\mathrm{Hom}_{\mf g}(E, U/I) = \dim\mathrm{Hom}_{\mf g}(E, \mc L(K(y\cdot \mu), K(y\cdot \mu))).\]
This means that Kostant's problem has positive solution for $K(\ov N)$.	 
\end{proof}

\begin{lem}\label{nlem17}
The supermodule $K(\overline{N})$ is projective in $\mc C_{K(\overline{N})}$. 
\end{lem}

\begin{proof}
Let $M\in \mc C_{K(\overline{N})}$. By adjunction,  we have
\begin{displaymath}
\text{Hom}_{\mf g}(K(\ov N), M)\cong \text{Hom}_{\mf g_{\geq 0}}(\ov N, (M)^{\mf g_1}).
\end{displaymath}
Our dominance assumptions on $\mu$ imply that 
\begin{displaymath}
\text{Hom}_{\mf g_{\geq 0}}(\ov N, (M)^{\mf g_1})\cong \text{Hom}_{\mf g_\oo}(\ov N, M)
\end{displaymath}
and the claim follows from the fact that $\ov N$ is projective in $\mathcal{C}_{\ov N}$.	
\end{proof}

We want to show that the functors
\begin{equation}\label{Eq::MainEquiv_1} 
F:= -\otimes_U K(\ov N): \mc H_I^1 \rightarrow \mc C_{K(\overline{N})},\qquad
G:=\mathcal{L}(K(\overline{N}),{}_-):\mc C_{K(\overline{N})}\to \mathcal{H}_I^1
\end{equation}
are mutually inverse equivalences. 

We have $G(K(\overline{N}))\cong U/I\in \mathcal{H}_I^1$. Moreover, from Lemma~\ref{nlem17}
it follows by the same arguments as in \cite[6.9(9)]{Ja83} that $G$ is exact.
As $G$ commutes with tensoring with finite dimensional $\mathfrak{g}$-supermodules and all
projectives in $C_{K(\overline{N})}$ have, by definition, the form $E\otimes K(\overline{N})$,
for some finite dimensional $\mathfrak{g}$-supermodule $E$, it follows that 
$G$ sends $C_{K(\overline{N})}$ to $\mathcal{H}_I^1$, in particular, $G$ is well-defined.

As in \cite[6.22]{Ja83}, the functor $F$ is left adjoint to $G$, in particular, 
$F$ is also well-defined. Using Lemma~\ref{nlem15}, the claim that $F$ and $G$ are mutually inverse equivalences of 
categories follows similarly to \cite[Theorem 5.9]{BeGe}, \cite[Theorem~5]{KhMa04} and \cite[Theorem~5.1]{MaMe12}.
\end{proof}

\subsection{Rough structure of Kac modules}

We denote by $\Xi:= \mc L(K(\ov N), -)\otimes_U K(y\cdot \mu)$ 
the equivalence from $\mc C_{K(\ov N)}$ to $\mc C_{K(y\cdot \mu)}$ 
in Theorem \ref{Thm::Equi}. The functor $\Xi$ induces a bijection between 
the sets $\text{Irr}(\mc C_{K(\ov N)})$ 
and $\text{Irr}(\mc C_{K(y\cdot \mu)})$ of isomorphism classes of 
simple objects in 
$\mc C_{K(\ov N)}$ and $\mc C_{K(y\cdot \mu)}$, respectively. 
We note that simple objects in $\mc C_{K(\ov N)}$ and $\mc C_{K(y\cdot \mu)}$
are not necessarily simple as $\mf g$-supermodules. However, just as in \cite[Section~11]{MaSt08},
every simple object in 
$\mc C_{K(\ov N)}$ and $\mc C_{K(y\cdot \mu)}$ has simple top, as a $\mf g$-supermodule,
and the annihilator of the radical of a simple object is strictly bigger than that 
of the simple top.
Consequently, we have an induced bijection
\[ \hat\Xi:~\text{Irr}^{\mf g}(\mc C_{K(\ov N)})\rightarrow
\text{Irr}^{\mf g}(\mc C_{K(y\cdot \mu)}) \]
between the sets of isomorphism classes of simple $\mf g$-supermodule quotients
of simple objects in $\mc C_{K(\ov N)}$ and $\mc C_{K(y\cdot \mu)}$.

For  $L(V) \in \text{Irr}^{\mf g}(\mc C_{K(\ov N)})$, we define $\xi_V\in \mf h^*$ via 
\begin{align} &L(\xi_V) \cong \hat\Xi(L(V)), \end{align} 
in particular, we have $\xi_{\ov N} = y\cdot \mu$ since $\Xi K(\ov N) =K(y\cdot \mu)$.
We are now in a position to state the main result of this section which describes 
rough structure of Kac modules.

\begin{cor}\label{thmrough}
For $L(V),~L(W) \in \emph{Irr}^{\mf g}(\mc C_{K(\ov N)})$, we have the following 
multiplicity formula in the category of $\mf g$-supermodules:
\begin{align}\label{Eq::RoughStr}
[K(V):L(W)] =[K(\xi_V):L(\xi_W)].
\end{align}
\end{cor}

\begin{proof}
The two sides of the equality are matched using $\Xi$,
cf.  \cite[Theorem~72]{MaSt08}.
\end{proof}

Theorem~\ref{thmrough} says that the combinatorics of the rough structure of $K(V)$
only depends on the annihilator of $V$.
Depending on $V$, the rough structure of $K(V)$ might coincide, or not, with its fine structure.
In general, just like in \cite[Section~11]{MaSt08}, our approach and,
in particular, Theorem~\ref{thmrough} does not allow us to control possible simple 
subquotients of $K(V)$ whose annihilator is strictly bigger than that of $K(V)$.
Moreover, it is known that this fine structure of $K(V)$ really depends on $V$
and not just on the annihilator of $V$. Furthermore, in general, there is also a chance that
the module $K(V)$ might be non-artinian.

\subsection{Rough structure of simple $\mathfrak{gl}(m|n)$-supermodules} 
In this subsection, we obtain a similar description of the rough structure 
of restrictions to $\mf g_\oo$ of simple $\mf g$-supermodules in 
$\text{Irr}^{\mf g}(\mc C_{K(\ov N)})$. 

\subsubsection{}
Let $\chi: = \chi_{\mu}$ be the $\mf g$-central character of the typical weight $\mu$. Then $K(\ov N) = K(\ov N)_{\chi }$ and $\ov N  = \ov N_{\chi^{\oo}}$ since  $K(y\cdot \mu)\in  \mf g\text{-mod}_{\chi }$ and $V(y\cdot \mu) \in  \mf g\text{-mod}_{\chi^{\oo}}$.

\begin{thm}\label{Thm::GorelikEquivThm}\emph{(}\cite[Theorem 1.3.1]{Go02s}\emph{)}
The categories $\mf g_{\oo}\emph{-smod}_{\chi^{\oo}}$ and $\mf g\emph{-smod}_{\chi}$ are equivalent 
via the equivalences $(\emph{Ind}_{\mf g_\oo}^{\mf g})_{\chi}$ and 
$(\emph{Res}_{\mf g_\oo}^{\mf g})_{\chi^{\oo}}$.
\end{thm} 

We recall the equivalence $\Xi_\oo$ between  $\mc C_{K(\ov N)}$ and $\mc C_{K(y\cdot \mu)}$ 
that was defined in Theorem \ref{thm::g0CokerEquiv}.   

\begin{lem} \label{nlem25}
There is an isomorphism of functors
$\emph{Res}_{\mf g_\oo}^{\mf g}\circ \Xi \cong \Xi_\oo \circ \emph{Res}_{\mf g_\oo}^{\mf g}$.
\end{lem}

\begin{proof}
Applying $(\cdot)_{\chi}$ to $\text{Ind}_{\mf g_\oo}^{\mf g}(V(y\cdot \mu)) \twoheadrightarrow K(y\cdot \mu)$,
we get $(\text{Ind}_{\mf g_\oo}^{\mf g}(V(y\cdot \mu)))_{\chi} \twoheadrightarrow K(y\cdot \mu)$
as $K(y\cdot \mu)$ is indecomposable.
Hence Theorem~\ref{Thm::GorelikEquivThm} gives 
$(\text{Ind}_{\mf g_\oo}^{\mf g}(V(y\cdot \mu)))_{\chi}\cong K(y\cdot \mu)$. Therefore, we have 
$ \text{Ind}_{\mf g_\oo}^{\mf g}(V(y\cdot \mu)) = K(y\cdot \mu)\oplus M$,
for some $\mf g$-supermodule $M$ with $M_{\chi} =0$. 
Similarly,  $\text{Res}_{\mf g_\oo}^{\mf g}(K(\ov N)) = \ov N \oplus M'$, for 
some $\mf g$-supermodule $M'$ with $M'_{\chi_{\oo}} =0$. 

Recall that $X\otimes_A Y=0$ provided that there is an element $a\in A$ which annihilates $X$
and acts invertibly on  $Y$, in particular, if $X$ and $Y$ have different generalized central characters. 
Taking this into account, the observations in the previous  paragraph allow us to compute as follows:
\begin{displaymath}
\begin{array}{rcl} 
\text{Res}_{\mf g_{\oo}}^{\mf g}\left(\mc L(K(\ov N),{}_-)\otimes_U K(y\cdot \mu)\right)
&\cong&\text{Res}_{\mf g_{\oo}}^{\mf g}\left(\mc L(K(\ov N),{}_-)\otimes_U  
\text{Ind}_{\mf g_\oo}^{\mf g}V(y\cdot \mu) \right)\\
&\cong&\mc L( \text{Res}_{\mf g_{\oo}}^{\mf g}(K(\ov N)),  
\text{Res}_{\mf g_{\oo}}^{\mf g}({}_-)) \otimes_{U_\oo} V(y\cdot \mu) \\
&\cong&\mc L(\ov N, \text{Res}_{\mf g_\oo}^{\mf g}({}_-))\otimes_{U_\oo} V(y\cdot \mu),
\end{array}
\end{displaymath}
as desired. Here in the second row we used the obvious  analogue of \cite[Lemma~3.7(2)]{Co16}.
\end{proof}

We have a bijection 
$$\hat\Xi_\oo : \text{Irr}^{\mf g_\oo}(\mc C_{\ov N})\rightarrow \text{Irr}^{\mf g_\oo}(\mc C_{V(y\cdot \mu)}),$$
induced by  $\Xi_\oo$,  between the sets of isomorphism classes  of the simple $\mf g_\oo$-quotients 
of simple objects in $\mc C_{\ov N}$ and $\mc C_{V(y\cdot \mu)}$. 
For a given $W \in \text{Irr}^{\mf g_\oo}(\mc C_{\ov N})$, we define  
the related weight $\zeta_W\in \mf h^*$ by 
\begin{align} 
&V(\zeta_W) \cong \hat\Xi_\oo(W).
\end{align} 
The next statement describes the $\mf g_\oo$-rough structure of simple 
$\mf g$-supermodules in terms of the combinatorics of category $\mathcal{O}$ for $\mf g_\oo$.

\begin{cor}
For $L(V)\in \mathrm{Irr}^{\mf g}(\mc C_{K(\ov N)})$ and 
$W\in \mathrm{Irr}^{\mf g_\oo}(\mc C_{\ov N})$, 
we have the following multiplicity formula in the category of $\mf g_\oo$-supermodules:
\begin{align}\label{Eq::RoughStr-s}
[\mathrm{Res}^{\mathfrak{g}}_{\mathfrak{g}_\oo}(L(V)):W] =
[\mathrm{Res}^{\mathfrak{g}}_{\mathfrak{g}_\oo}(L(\xi_V)):V(\zeta_W)].
\end{align}
\end{cor}

\begin{proof}
Equality~\eqref{Eq::RoughStr-s} is obtained by, first, applying $\Xi_\oo$ to the left hand side 
and then using Lemma~\ref{nlem25}.
\end{proof}
 
%%%%%%%%%%%%%%

\vspace{5mm}

Department of Mathematics, Uppsala University, Box. 480,
SE-75106, Uppsala, SWEDEN,\\
emails: {\tt chih-whi.chen\symbol{64}math.uu.se}\hspace{1cm} {\tt mazor\symbol{64}math.uu.se}

\end{document}